\documentclass[aps,prx,reprint,superscriptaddress]{revtex4-2}
\usepackage{graphicx}% Include figure files
\usepackage{dcolumn}% Align table columns on decimal point
\usepackage{bm}% bold math
\usepackage{amssymb,amstext,amsmath}
\usepackage{xr}
\usepackage{color}
	 \definecolor{darkred}{rgb}{0.75,0,0}
	 \definecolor{darkgreen}{rgb}{0,0.5,0}
	 \definecolor{darkblue}{rgb}{0,0,0.75}
  	 \definecolor{darkorange}{rgb}{1,0.9,0.1}
	 \definecolor{dark}{rgb}{0,0,0}

\newtheorem{theorem}{Theorem}

\newtheorem{definition}{Definition}
\newtheorem{remark}{Remark}
\newtheorem{proof}{Proof}
\newtheorem{lemma}{Lemma}
\newtheorem{corollary}{Corollary}

\newtheorem{assumption}{Assumption}
\newtheorem{example}{Example}
\begin{document}

\preprint{APS/123-QED}

\title{Convergence rate of opinion dynamics with complex interaction types}

\author{Lingling Yao}
\affiliation{%
Center for Systems and Control, College of Engineering, Peking University, Beijing 100871, China}
\author{Aming Li}
\thanks{Corresponding author: amingli@pku.edu.cn}
\affiliation{%
Center for Systems and Control, College of Engineering, Peking University, Beijing 100871, China}
\affiliation{
Center for Multi-Agent Research, Institute for Artificial Intelligence, Peking University, Beijing 100871, China}

%\author{Lingling Yao}
%\affiliation{%
%Center for Systems and Control, College of Engineering, Peking University, Beijing 100871,  People's Republic of China}
%\author{Aming Li}
%\thanks{Corresponding author: amingli@pku.edu.cn}
%\affiliation{%
%Center for Systems and Control, College of Engineering, Peking University, Beijing 100871, People's Republic of China}

%\affiliation{
%Center for Multi-Agent Research, Institute for Artificial Intelligence, Peking University, Beijing 100871, People's Republic of China}

% \collaboration{CLEO Collaboration}%\noaffiliation

\date{\today}% It is always \today, today,

\begin{abstract}
The convergence rate is a crucial issue in opinion dynamics, which characterizes how quickly opinions reach a consensus and tells when the collective behavior can be formed. However, the key factors that determine the convergence rate of opinions are elusive, especially when individuals interact with complex interaction types such as friend/foe, ally/adversary, or trust/mistrust. In this paper, using random matrix theory and low-rank perturbation theory, we present a new body of theory to comprehensively study the convergence rate of opinion dynamics. First, we divide the complex interaction types into five typical scenarios: mutual trust $(+/+)$, mutual mistrust $(-/-)$, trust$/$mistrust $(+/-)$, unilateral trust $(+/0)$, and unilateral mistrust $(-/0)$. For diverse interaction types, we derive the mathematical expression of the convergence rate, and further establish the direct connection between the convergence rate and population size, the density of interactions (network connectivity), and individuals' self-confidence level. Second, taking advantage of these connections, we prove that for the $(+/+)$, $(+/-)$, $(+/0)$, and random mixture of different interaction types, the convergence rate is proportional to the population size and network connectivity, while it is inversely proportional to the individuals' self-confidence level. However, for the $(-/-)$ and $(-/0)$ scenarios, we draw the exact opposite conclusions. Third, for the $(+/+,-/-)$ and $(-/-,-/0)$ scenarios, we derive the optimal proportion of different interaction types to ensure the fast convergence of opinions. Finally, simulation examples are provided to illustrate the effectiveness and robustness of our theoretical findings.
\end{abstract}

\maketitle

\section{Introduction}
Over the last few years, the investigation of reaching a convergence among a group of agents has attracted remarkable attention from many fields, such as control theory \cite{Cao, Ols,Xu, Ang, Ji, Hong, Li}, ecology \cite{Rey, Vic}, sociology \cite{Pros17, Pros172, Xia:20, Wang22}. To understand and analyze the underlying reasons for such limiting group behavior, several mathematical models have been proposed from the perspective of opinion dynamics, including the DeGroot, the Abelson, the Friedkin-Johnsen, the bounded confidence, and the Altafini models \cite{De:74, Fri99, Altafini18, Tia18}. The seminal discrete-time DeGroot model assumes that each individual's opinion at the next time step is a weighted average of his/her current opinion and those of his/her neighbors. Based on some properties of infinite products of stochastic matrices, it is further proved that individuals' opinions can achieve consensus in the sense that all individuals agree
upon certain quantities of interest under the DeGroot model \cite{Wan18}.
Since the DeGroot model reflects the fundamental human cognitive capability of taking convex combinations when integrating related information, it has been extensively studied from various perspectives, including belief system theory \cite{Fri16, Yang:23}, social power theory \cite{Chen, Tia22}, and pluralistic ignorance theory \cite{Ye19, Liu23}.

In social networks, competition, antagonism, and mistrust between individuals and their groups are ubiquitous in many antagonistic systems describing bimodal coalitions, like two-party political systems, duopolistic markets, rival business cartels and competing international alliances. To deal with such situations, the signed networks is introduced, where the signs indicate the social relationships between each individual and his/her neighbors --- a positive sign and a negative sign represent trust (or friendship) and mistrust (or antagonism), respectively \cite{Altafini:13, Cartwright}. The structural balance theory offers a vital analytical tool for examining opinion dynamics on signed networks, forging a link between these networks and their corresponding unsigned networks (all weights in this network are non-negative). It depicts the balance between trust and mistrust that dictates individuals' opinions to become
closer or further apart, respectively \cite{Cartwright}. Based on the structural balance theory, an important opinion dynamics model proposed by Altafini has attracted increasing attention lately \cite{Altafini:13}. Different from the consensus of opinions for all individuals under the DeGroot model, it demonstrates that the agents can achieve a form of ``bipartite consensus'', where all agents converge to a value which is the same for all in modulus but not in sign \cite{Altafini:13}. Additionally, the discrete-time counterpart Altafini model has been extensively studied in \cite{Xia:16, Liu17}. At present, the Altafini model and its extensions have been considerably investigated and some insightful results have been derived based on the Altafini model with features such as switching network, time-varying network, quasi-structurally balanced network \cite{Ame:17, C:20, Prosk:15, Meng:16, Meng:17, Shi, Liu19, Liu21}.

The convergence rate is a fundamental indicator to evaluate the system performance, which provides many meaningful instructions for decision-making and engineering practice \cite{Ding:22, Gha14, Ned19}.
For example, in the decision-making process, an emergent group opinion or consensus often has to be made within a short finite time. In the engineering field, accelerated algorithms can increase efficiency, reduce energy consumption, and optimize performance. Therefore, some researchers have shown particular interest in it and made some remarkable progress. It is pointed out that the convergence speed of the first-order continuous-time system is generally governed by the minimum non-zero eigenvalue (also termed as ``algebraic connectivity'') of the corresponding Laplacian matrix for undirected communication graph \cite{Olfati07}, while the minimum real part of the non-zero eigenvalue of the Laplacian matrix for directed communication graph \cite{Yu10}.

The existing literature mainly deals with accelerated convergence from two points of view. One is to identify the optimal network topology limited by the structured constraints to maximize the convergence rate for a given protocol (e.g., optimizing the weight matrix \cite{Xiao, Apers}). It is shown that if the network topology is symmetric, the problem of finding the fastest converging linear iteration can be cast as a semidefinite programming problem in \cite{Xiao}. For the directed acyclic graphs, the convergence rate can be enhanced by adding some certain edges, as suggested in \cite{Zhang17}. The other one is to seek the optimal protocol to improve the convergence rate for a given network topology (e.g., utilizing finite-time control \cite{Wang10, Xiao14}, employing the graph signal processing \cite{Yi20}, applying the Laplacian matrix-valued functions method \cite{Qin, Ma23} and introducing the individual's memory \cite{Mariano, Yi}). Please refer to \cite{Zhang14, Wang20} for more accelerated algorithms.

%According to the notion of the second largest eigenvalue modulus, we define the convergence rate $r$ as a logarithmic function and variable is the largest or the second largest eigenvalue modulus of interpersonal influence matrix $W$, which ensures that $r>0$ ($r$ is well defined) and the larger variable is, the slower convergence rate. taking into account factors such as population size, network connectivity, self-confidence level (the influence of an individual's current opinion on his/her opinion at the next moment), and interaction types?
The aforementioned works have provided invaluable insights and effective solutions for the accelerated convergence problem under their respective assumptions and formulations. However, the influence that various communication networks have on the convergence rate remains unclear, especially in terms of the complex interaction types. While solving characteristic equations offers a numerical method for studying the consensus rate, this approach becomes increasingly challenging for large-scale populations due to the high time and space complexity. Hence, the study on the spectrum of the adjacency matrix or Laplacian matrix corresponding to the signed network is still a significant theoretical challenge. Most importantly, there is a lack of understanding of the explicit mechanism for the effect of graph variation on convergence rate. As a result, a systematic exploration of the opinion consensus rate for large-scale populations is still lacking.

Motivated by previous discussions, we aim to develop a framework for analyzing the convergence rate of opinion dynamics with complex interaction types. To achieve this goal, the following three key problems must be addressed: a). What are the key factors that affect the convergence rate?
b). How can we establish the quantitative relationship between the convergence rate and these factors? c). What are the specific impacts of these factors on the convergence rate? The main contributions of this article are summarized as follows.
\begin{itemize}
\item  According to the trust and mistrust relationships between individuals, the signed interaction types are categorized into five scenarios: $(+/+)$, $(-/-)$, $(+/-)$, $(+/0)$, $(-/0)$. Subsequently, two random signed networks are constructed in Subsections \ref{subsec1} and \ref{subsec2}. Furthermore, in Subsection \ref{subsec3}, a novel opinion dynamics model on these signed networks is proposed, which can be viewed as a generalized version of the Altafini model with large-scale populations.

\item By structural balance theory, we find that the convergence rate of the generalized Altafini model is governed by $\rho(W)$ or $\rho_{2}(W)$ defined in (\ref{equ_r1}) (see Lemma \ref{lem7}). With the aid of random matrix theory and low-rank perturbation theory, we present the quantitative convergence rate via the estimation of eigenvalues, thereby establishing a direct connection between the convergence rate and
    some key factors in Theorems \ref{the1} and \ref{the2}.
  \item To further tackle the convergence rate issue of our model, we derive the monotonicity of convergence rate $r$ with respect to the population size $n$, network connectivity $P$, and individuals' self-confidence level $d$ in Corollary \ref{cor1} and Theorem \ref{the3}, where for $(-/-)$ and $(-/0)$ scenarios, it behaves completely opposite to the other scenarios.

 \item  The effects of $(-/-)$ and $(-/0)$ on the convergence rate are further considered by analyzing the $(+/+,-/-)$ and $(-/-,-/0)$ scenarios. For the $(+/+,-/-)$ scenario, when two interaction types have approximately the same proportion, the system achieves the fastest convergence. For the $(-/-,-/0)$ scenario, the convergence rate is inversely proportional to the proportion of $(-/-)$ interaction type (see Theorems \ref{the4} and \ref{the5}).
\end{itemize}

The rest of this paper is organized as follows. Section \ref{pre1} offers some useful preliminary knowledge of signed graphs and the construction of networks with diverse types considered in this paper. In Section \ref{pro1}, we introduce the opinion dynamics model with large-scale populations along with some useful lemmas. Moreover, some convergence analyses are given to lay the groundwork for the following research on convergence rate. In Section \ref{main1}, we present our main theoretical results. The theoretical results are verified by numerical simulations in Section \ref{main2}. Finally, Section \ref{con1} concludes this paper. The notations and abbreviations used in this paper are listed in TABLE \ref{tab1}.
\begin{table}[ht]\label{tab1}
\centering
\caption{Notations}
\begin{tabular}{p{2cm}p{6.3cm}}
\hline Symbols & Definitions \\
\hline
$\mathbb{C}$ & set of complex numbers \\
$\mathbb{R}$ & set of real numbers \\
$\mathbb{Z}$ & set of integers \\
$\mathbb{R}^{m\times n}$ & set of $m\times n$ real matrices \\
$O$ & any zero matrix with proper dimension \\
$\mathbf{1}_{n}$ & $[1,\dots,1]^T$ \\
$\text{diag}\{a\}$ & diagonal matrix with diagonal element $a$ \\
$A_{ij}$ & entry at the $i$-th row and the $j$-th column of matrix $A$ \\
$|A|$ & a nonnegative matrix in which each element $|A|_{ij}$ equals $|A_{ij}|$\\
$|A|_{i}$ & the sum of the elements in the $i$-th row of matrix $|A|$ \\
$\lambda_{i}(A)$ & the $i$-th eigenvalue of matrix $A$ \\
$\text{Re}(\lambda_{i}(A))$ & real part of $\lambda_{i}(A)$ \\
$\rho(A)$ & spectral radius of matrix $A$ \\
$[n]$ & set of $1,2,\cdots,n$ \\
$\text{fix}(x)$ & $\text{fix}(x)$ rounds $x$ to the nearest integer toward zero \\
$a:b:c$ & a sequence $a,a+b,a+2b,\cdots,a+\text{fix}((c-a)/b)*b$
 \\
\hline
\end{tabular}
\end{table}
%\textbf{Notations}: Throughout this paper, let $\mathbb{C}$ and $\mathbb{R}$ be the sets of complex and real numbers, respectively. $\mathbb{R}^{m\times n}$ denotes the set of $m\times n$ real matrices. Moreover, let $\mathbb{Z}$ denote the set of integers. $O$ represents any zero matrix with proper dimension and $\mathbf{1}_{n}=[1,\dots,1]^T.$ $\text{diag}\{a\}$ denotes the diagonal matrix with diagonal element $a$. $|A|\in\mathbb{R}^{n\times n}$  stands for a nonnegative matrix in which each element $|A|_{ij}$ is the absolute value of $A_{ij}$. The $i$-th row's sum of matrix $|A|$ is represented by $|A_{i}|$. $\lambda_{i}(A)$ represents an eigenvalue of $A$ with $\text{Re}(\lambda_{i}(A))$ being its real part and $\rho(A)$ being its spectral radius. Besides, matrices, if their dimensions are not explicitly stated, are assumed to be compatible with algebraic operations. The notation $a:b:c$ denotes the set consisting of values that range from $a$ to $c$, with values taken every $b$, $a,b,c\in\mathbb{Z}$.
\section{Preliminaries}\label{pre1}
In this section, we briefly review some basic concepts of graph theory used in later sections. Then, we present the construction method for two signed interaction networks and formulate the opinion dynamics model on these networks.
\subsection{Signed Graph}
 Let $\mathcal{G}(W)=\{\mathcal{V},\mathcal{E},W\}$ denote a weighted signed graph of order $n$, with the nodes set $\mathcal{V}=\{v_{1},\dots, v_{n}\}$, the edges set $\mathcal{E}=\mathcal{V}\times \mathcal{V}$, and the adjacency matrix $W=[W_{ij}]$. An edge $e_{ij}=(v_{i}, v_{j})\in\mathcal{E}$ means that node $j$ can get information from node $i$. $W_{ij}\neq0\Leftrightarrow\left(v_j, v_i\right) \in \mathcal{E}$. $W_{ij}\neq 0$ if and only if $e_{ji}\in\mathcal{E}$.  A signed digraph $\mathcal{G}(W)$ is structurally balanced if there exists a bipartition $\left\{\mathcal{V}_1, \mathcal{V}_2\right\}$ of the vertices, where $\mathcal{V}_1 \cup \mathcal{V}_2=\mathcal{V}$ and $\mathcal{V}_1 \cap \mathcal{V}_2=\varnothing$, such that $W_{i j} \geq 0$ for $\forall v_i, v_j \in \mathcal{V}_l$ $(l \in\{1,2\})$ and $W_{i j} \leq 0$ for $\forall v_i \in \mathcal{V}_l, v_j \in$ $\mathcal{V}_q, l \neq q$ $(l, q \in\{1,2\})$; and $\mathcal{G}$ is structurally unbalanced otherwise\cite{Altafini:13}. A walk of length $k$ from node $i$ to $j$ is a sequence of nodes $i_{0},\dots, i_{k}\in\mathcal{V}$, where $i_{0}=i$, $i_{k}=j$. Especially, a walk from a node to itself is a cycle. A signed directed cycle with an even/odd number of edges having negative weights is called a positive/negative directed cycle. Agent $j$ is a reachable node of agent $i$ if there exists a walk from agent $i$ to agent $j$. A graph is periodic if it has at least one cycle and the length of any cycle is divided by some integer $h>1$. Otherwise, a graph is called aperiodic. A strongly connected subgraph $\mathcal{G}^{'}$ of digraph $\mathcal{G}$ is called a strongly connected component (SCC) if it is not contained by any larger strongly connected subgraph. A SCC without incoming arcs from other SCCs is called a closed SCC (CSCC); otherwise, it is called an open SCC (OSCC).

In the following subsections \ref{subsec1} and \ref{subsec2}, we will outline the construction method for two types of signed interaction networks, which lays an important foundation for our subsequent research.
\subsection{Random mixture interactions}\label{subsec1}
  For random mixture interactions of various interaction types, we construct the interaction network with $n$ individuals in the following way:
   \\
   i) individuals $i$ and $j$ interact with probability $P\neq0$;
   \\
   ii) the interaction strength $S_{ij}$ and $S_{ji}$ take the value of a random variable $Z$ with mean 0 and variance $\sigma^2$ independently.

   $S_{i j}<0$ ($S_{i j}>0$) represents the mistrust/trust that individual $j$ has for individual $i$, and $S_{i j}=0$ denotes that the interaction strength from individual $j$ to individual $i$ is zero.

  For simplicity, we refer to the interaction network above as the random mixture interaction network. On the basis of the construction method of the interaction network, we obtain some statistics of the interaction matrix $S$. Specifically,
  \[
  \left\{\begin{array}{l}
\mathbb{E}\left(S_{ij}\right)=P\mathbb{E}\left(Z\right)=0, \\
\mathbb{E}\left(|S_{i j}|\right)=P\mathbb{E}\left(|Z|\right),\\
\mathbb{E}\left(S_{ij}^{2}\right)=P\mathbb{E}\left(Z^{2}\right)=P\sigma^2, \\
\operatorname{Var}\left(S_{i j}\right)=\mathbb{E}\left(S_{ij}^{2}\right)-\mathbb{E}^{2}\left(S_{ij}\right)=P\sigma^2.
\end{array}\right.
\]
For large $n$, since $|S_{ij}|$ is i.i.d. chosen from the distribution of $|Z|$, the $i$-th row sum $|S|_{i}$ of matrix $|S|$ is roughly a constant
\begin{equation}\label{equ_30}\sum\limits_{j=1}^n |S_{i j}| \approx(n-1) \mathbb{E}\left(|S_{i j}|\right)=(n-1)P\mathbb{E}(|Z|).\end{equation}

In the random mixture network, we consider the density of interactions and the interaction strength of different interaction types. However, we cannot distinguish mistrust and trust interactions from the above network and systematically determine the effect of interaction types on the convergence rate. This motivates us to further construct an interaction network with a certain proportion of five interaction types as detailed in Subsection \ref{subsec2}.
  \subsection{Complex mixture interactions}\label{subsec2}
For the mixture interactions under a certain proportion of five interaction types, we construct the interaction network with $n$ individuals in the following way: i) individuals $i$ and $j$ interact with probability $P\neq0$; ii) the interaction strengths are categorized into five typical scenarios (see FIG. \ref{pic-fuhao}):
 \\
 (1) Mutual trust $(+/+)$ interaction with proportion $P_{+/+}$. The interaction strengths $S_{ij}$ and $S_{ji}$ take the values of $|Z|$  independently.
 \\
 (2) Mutual mistrust $(-/-)$ interaction  with proportion $P_{-/-}$. The interaction strengths $S_{ij}$ and $S_{ji}$ take the values of $-|Z|$  independently.
 \\
 (3) Trust$/$mistrust $(+/-)$ interaction with proportion $P_{+/-}$. The interaction strengths $S_{ij}$ and $S_{ji}$ have opposite signs: one takes the value of $|Z|$ while the other takes the value of  $-|Z|$.
 \\
 (4) Unilateral trust $(+/0)$ interaction with proportion $P_{+/0}$. One of the interaction strengths $S_{ij}$ and $S_{ji}$ takes the value of $|Z|$ while the other takes the value of 0.
 \\
 (5) Unilateral mistrust $(-/0)$ interaction with proportion $P_{-/0}$. One of the interaction strengths $S_{ij}$ and $S_{ji}$ takes the value of $-|Z|$ while the other takes the value of 0.
\begin{figure}
    \centering
   \includegraphics[scale=0.5]{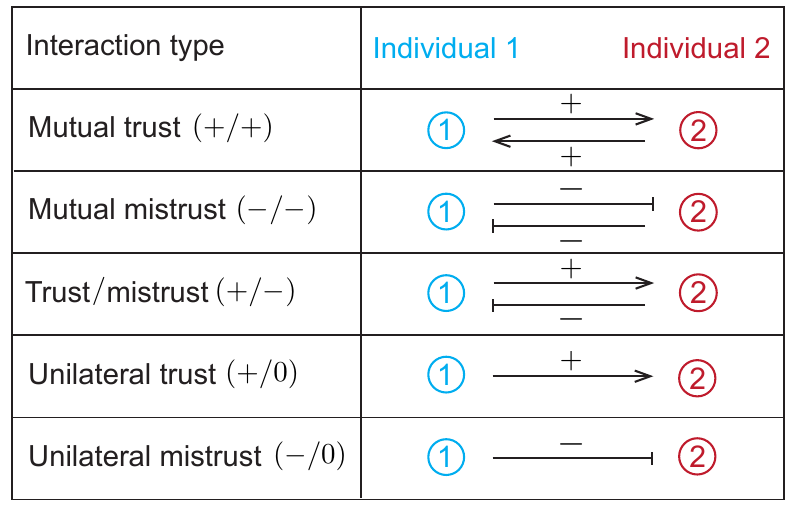}%{wuzhong.pdf}
    \caption{Five interaction types between individuals 1 and 2: mutual trust $(+/+)$, mutual mistrust $(-/-)$, trust$/$mistrust $(+/-)$, unilateral trust $(+/0)$, unilateral mistrust $(-/0)$.}
    \label{pic-fuhao}
    \end{figure}

 For simplicity, we refer to the interaction network above as the complex mixture interaction network. Then, we can obtain some statistics of the interaction matrix $S$. Specifically, we have
 \begin{equation*}\left\{\begin{array}{l}
\mathbb{E}\left(S_{ij}\right)=P\bar{P}\mathbb{E}\left(|Z|\right), \\
\mathbb{E}\left(|S_{i j}|\right)=P\hat{P}\mathbb{E}\left(|Z|\right),\\
\mathbb{E}\left(S_{ij}^{2}\right)=\displaystyle P\hat{P}\mathbb{E}\left(|Z|^{2}\right)=P\hat{P}\sigma^2, \\
\mathbb{E}\left(S_{i j}S_{j i}\right)=PP^{*}\mathbb{E}^{2}\left(|Z|\right),
\end{array}\right. \end{equation*}%\label{equ_28}\label{equ_29}
where \begin{equation}\label{equ_28}\hat{P}=P_{+/+}+P_{+/-}+P_{-/-}+\frac{1}{2}P_{+/0}+\frac{1}{2}P_{-/0},\end{equation} \begin{equation}\label{equ_29}\bar{P}=P_{+/+}-P_{-/-}+\frac{1}{2}P_{+/0}-\frac{1}{2}P_{-/0},\end{equation}
and \begin{equation}\label{equ_299}P^{*}=P_{+/+}+P_{-/-}-P_{+/-},\end{equation}
where $0\leq\hat{P}\leq1$, $-1\leq\bar{P}\leq1$ and $\bar{P}\leq\hat{P}$.

Similar to random mixture interactions, for large $n$, the $i$-th row sum $|S|_{i}$ of matrix $|S|$ approaches to a constant
\begin{align}\label{equ_31}\sum\limits_{j=1}^n |S_{i j}|&\approx(n-1) \mathbb{E}\left(|S_{i j}|\right)=(n-1)P\hat{P}\mathbb{E}(|Z|).\end{align}
In order to clarity the $i$-th row sum $|S|_{i}$ for random mixture interactions and that for complex mixture interactions, we denote them as $C_{\text{r}}$ and $C_{\text{m}}$, respectively.

Next, we formulate the opinion dynamics model on these networks presented in Subsections \ref{subsec1} and \ref{subsec2}.
\subsection{Model formulation}\label{subsec3}
  Consider a social network composed of $n$ individuals discussing one topic simultaneously. The opinion of individual $i$ at time $k$ is represented by $X_i(k)\in[-1,1]$. $X_i(k)>0$ $(X_i(k)<0)$ denotes the support (rejection) of individual $i$, and $X_i(k)=0$ represents a neutral attitude. The magnitude of $X_i(k)$ indicates the strength of attitude, where $\left|X_i(k)\right|=1$ represents the maximal support or rejection.

Motivated by the Altafini model proposed in \cite{Altafini:13}, suppose the opinion of the $i$-th individual evolves as
  \begin{equation}\label{equ1}X_i(k+1)=\sum_{j=1}^n W_{i j}X_j(k),\end{equation}
   where
   \[
   W_{i j}=\left\{\begin{array}{l}
  \displaystyle\frac{S_{ij}}{d+|S_{i}|},  \text{if} \ i\neq j,\\
   \displaystyle\frac{d}{d+|S_{i}|},  \text{otherwise}.\\
   \end{array}\right.
   \]
 $W_{i j}\in[-1,1]$ represents the weighted average influence that individual $j$ has on individual $i$, and $d>0$ denotes the self-confidence level. Furthermore, let $X(k+1)=[X_1(k+1),\cdots,X_n(k+1)]^{T}$, then the system (\ref{equ1}) can be written as
   \begin{equation}\label{equ001}X(k+1)=WX(k).\end{equation}
  % where for large $n$, $\sum\limits_{j=1}^{n}|W_{ij}|=1$.
   %In this paper, we suppose each individual is enough open to others' opinions, i.e.,  $W_{i i}\ll1$ and $d\ll n$.

\begin{remark}\textup{
By equations (\ref{equ_30}) and (\ref{equ_31}), the $i$-th row sum of matrix $|S|$ is roughly a constant that does not depend on individual $i$. Therefore, for random mixture interactions and complex mixture interactions, the equality $W_{11}=W_{22}=\cdots=W_{nn}$ always holds in system (\ref{equ001}).}
\end{remark}

\begin{remark}
\textup{According to equation (\ref{equ1}), we have $\sum_{j=1}^{n}|W_{ij}|=1$ for large $n$. Thus, if $X_{i}(0) \in [-1,1]$, we have $X_{i}(k) \in [-1,1]$, $\forall i\in[n]$. In summary, system (\ref{equ001}) can be regarded as a generalized discrete-time Altafini model when $n$ is sufficiently large.}
\end{remark}

\begin{definition}(Convergence and consensus)
For large $n$, system (\ref{equ001}) is said to converge if $\forall X(0)\in\mathbb{R}$, the limit $\lim\limits_{k \to +\infty}X(k)$ exists. Moreover, it admits consensus if $\lim\limits_{k \rightarrow \infty}\left|X_i(k)-X_j(k)\right|=0$. If $\lim\limits_{k \rightarrow \infty}\left|X_i(k)\right|=\alpha>0$, it is said to reach the bipartite consensus, $\alpha\in[-1,1]$, $\forall i, j\in[n]$. If $\lim\limits_{k \rightarrow \infty}X_i(k)=0,$ it is stable.
\end{definition}

%\begin{remark}
%By matrix theory, system (\ref{equ001}) converges if and only if matrix $W$ is regular\footnote{Regular matrix $W$ means that the limit $\lim\limits_{k \to +\infty}W^k$ exists.}.
%\end{remark}
\subsection{Some basic Lemmas}

\begin{lemma}\label{lem1} \textup{(See Better Theorem in \cite{Horn:85}) Let $\lambda\in\mathbb{R}$ and $x\in\mathbb{R}^{n}$ be an eigenpair of matrix $A\in\mathbb{R}^{n\times n}$ such that $\lambda$ satisfies the following inequalities \[
\left|\lambda-a_{i i}\right| \geq R_{i}^{\prime}=\sum\limits_{j \neq i}\left|a_{i j}\right|, \ \forall i=1, \ldots, n.
 \] If digraph $\mathcal{G}(A)$ is strongly connected, then every Gersgorin Circle passes through $\lambda$.}
\end{lemma}

%\begin{lem}\label{lem2} \textup{(See Theorem 1 in \cite{Xia:16}) Consider an irreducible $W$ with $|W|$ being stochastic. Assume that the graph $\mathcal{G}(W)$ has at least one negative edge. If $\mathcal{G}(W)$ is structurally unbalanced, then the discrete-time Altafini model converges to 0 for every initial value.} \end{lem}

According to Girko's Circular law in \cite{Girko:85} and Ellipse law in \cite{Sommer:88,Tao:10,All:12}, Lemma \ref{lem3} and Lemma \ref{lem4} are given as below:

\begin{lemma}\label{lem3} \textup{(See Theorem 1.10 (Circular law) in \cite{Tao:10}) Let matrix $W$ be the $n\times n$ random matrix whose entries $W_{ij}$ are i.i.d. random variables with mean zero and variance $\frac{1}{n}$. Then, when $n$ goes to infinity, the spectral distribution of $W$ converges (both in probability and the almost sure sense) to the uniform distribution on the unit disk.} \end{lemma}

\begin{lemma}\textup{(See \cite{All:12})\label{lem4} Let matrix $W$ be the $n\times n$ random matrix whose entries $W_{ij}$ are random variables with mean zero and variance $\frac{1}{n}$. The asymmetric entries of random matrix $W$ are i.i.d, and symmetric entries obey mean $\frac{z}{n}$. Then, when $n \rightarrow \infty$, the spectral distribution of $W$ converges to the uniform distribution on the complex plane centered at the origin, whose horizontal half-axis length is $1+z$ and the vertical half-axis length is $1-z$, i.e.,
\[
{\left(\frac{x}{1+z}\right)}^{2}+{\left(\frac{y}{1-z}\right)}^{2}\leq 1.
\]}
 \end{lemma}

\begin{lemma}\label{lem5} \textup{Consider an ellipse \[
\displaystyle{\left(\frac{x+c}{1+N_{a}}\right)}^{2}+{\left(\frac{y}{1-N_{b}}\right)}^{2}\leq 1,
\]
 then \[
 \max(|Q_{\text{leftmost }}|, |Q_{\text{rightmost }}|,|Q_{\text{uppermost }}|)=1+N_{a}+|c|,
 \]
 where $N_{a}>N_{b}>0$, $c\in\mathbb{R}$. $Q_{\text{leftmost}}$, $Q_{\text{rightmost}}$ and $Q_{\text{uppermost}}$ represent the leftmost, rightmost and uppermost points of above ellipse, respectively.}
  \end{lemma}
  \begin{proof} Substituting  $y=0$ into the boundary equation of this ellipse, we have \[
  Q_{\text{leftmost}}=(-1-N_{a}-c, 0)
  \] and
  \[
  Q_{\text{rightmost}}=(1+N_{a}-c, 0).
  \]
  If $c\leq 0$, then

\vspace{-0.5cm}
  \begin{align*}|Q_{\text{leftmost}}|&=|-1-N_{a}+|c||\\&=1+N_{a}+|c|\\&\geq |1+N_{a}-c|=|Q_{\text{rightmost}}|.\end{align*}
    If $c>0$, then

\vspace{-0.5cm}
\begin{align*}|Q_{\text{leftmost}}|&=|-1-N_{a}-c|\\&<1+N_{a}+|c|=|Q_{\text{rightmost}}|.\end{align*} Thus, we have \[
\max(|Q_{\text{leftmost}}|, |Q_{\text{rightmost }}|)=1+N_{a}+|c|.
 \]
 Furthermore, Substituting  $x=-c$ into the boundary equation of this ellipse, we have \[
 Q_{\text{uppermost}}=(-c, 1-N_{b})
 \] and

\vspace{-0.5cm}
  \begin{align*}|Q_{\text{uppermost}}|&=\sqrt{(1-N_{b})^2+c^2}\\ &\leq\sqrt{(1+N_{a})^2+c^2}\\&\leq\sqrt{(1+N_{a}+|c|)^2}\\&=1+N_{a}+|c|.
  \end{align*}
  Hence, this lemma holds.
  \end{proof}

\section{Problem setup}\label{pro1}
In this section, following a similar analysis to that used for the first-order continuous-time system as presented in \cite{Olfati07, Yu10}, we consider the convergence rate by examining the eigenvalues of matrix $W$ in system (\ref{equ1}).

Let matrix $W$ be transformed into the ``canonic'' form as follows:
\[
W=\left[\begin{array}{ccccc}
W_{11} & * & * & \ldots & * \\
 O & W_{22} & * & & * \\
 O &  O & W_{33} & & * \\
\vdots & \vdots& \vdots & \ddots & \vdots \\
 O &  O &  O & \ldots & W_{s s}
\end{array}\right],
\]
where each block matrix $W_{ii}$ is irreducible, and $\text{Eig}(W)$ and $\text{Eig}(W_{ii})$ are the sets of eigenvalues of $W$ and $W_{ii}$, respectively. Moreover, $\text{Eig}(W)=\bigcup_{i=1}^{s} \text{Eig}(W_{ii})$. The modulus of its eigenvalues can be arranged in decreasing order \begin{equation}\label{equ_r1}\rho(W)\geq\rho_{2}(W)\geq\cdots\geq\rho_{m}(W), 1<m\leq n,\end{equation} where $\rho_{i}(W)$ denotes the $i$-th
largest modulus of eigenvalues. Especially, the eigenvalue $\lambda_{2}$ satisfying $|\lambda_{2}|=\rho_{2}(W)$ is called the second-largest modulus eigenvalue of $W$.

\begin{lemma}\label{lem7}
\textup{For random mixture interactions and complex mixture interactions, the system (\ref{equ001}) achieves convergence. Furthermore, the convergence rate $r$ of system (\ref{equ001}) is
\[
\left\{\begin{array}{l}
   \text{-log}(\rho(W)), \text{if} \ \rho(W)<1,\\
   \text{-log}(\rho_{2}(W)), \text{otherwise},
   \end{array}\right.
   \]
   where $\rho_{2}(W)$ is defined in equation (\ref{equ_r1}).}
\end{lemma}
\begin{proof} By Lemma \ref{lem1}, if $\mathcal{G}(W_{ii})$ is an OSCC, $\rho(W_{ii})<1$.
 If $\mathcal{G}(W_{ii})$ is a CSCC, the distribution of eigenvalues of $W$ is divided into two cases:

1). If $\mathcal{G}(W_{ii})$ is structurally balanced, there exists a nonsingular matrix $S$ such that $\hat{W}_{ii}=S^{-1}W_{ii}S$ and $\hat{W}_{ii}$ is a nonnegative irreducible matrix. By Perron-Frobenius Theorem, $\exists |\lambda_{j}(W_{ii})|=1 \Leftrightarrow\lambda_{j}(W_{ii})=1$ and 1 is an algebraically simple eigenvalue of $W_{ii}$;

2). If $\mathcal{G}(W_{ii})$ is structurally unbalanced, by Theorem 1 in \cite{Xia:16}, we have $\rho(W_{ii})<1$.
Therefore, system (\ref{equ001}) can achieve convergence. Moreover, if 1 is an eigenvalue of $W$, then it is semisimple. Note that here semisimple means that the geometric and algebraic multiplicities are
the same, i.e., all Jordan blocks of the eigenvalue 1 are 1 by 1.

If $\rho(W)=1$, it follows that $W^k$ can be rewritten in the following form:

\vspace{-0.5cm}
   \begin{align}
   W^{k}&=Q^{-1}J^{k}Q\nonumber\\
   &=Q^{-1}\begin{bmatrix}
   J_{1} & O\\
   O   &  J_{2}
   \end{bmatrix}^{k}Q\nonumber\\
   &=\begin{bmatrix}Q_{1}^{-1} & Q_{2}^{-1} \end{bmatrix}\begin{bmatrix}
   J_{1} & O\\
   O   &  J_{2}
   \end{bmatrix}^{k}\begin{bmatrix}Q_{1}\\ Q_{2} \end{bmatrix}\nonumber\\
   &=Q_{1}^{-1}J_{1}^{k}Q_{1}+Q_{2}^{-1}J_{2}^{k}Q_{2},
   \nonumber\end{align}
   where $J_{1}$ is a matrix composed of all Jordan block matrices corresponding to eigenvalue 1. $Q_{1}^{-1}$ and $Q_{1}$ consist of all right eigenvectors and left eigenvectors corresponding to eigenvalue 1, respectively.

   For one Jordan block matrix $J_{i}\in\mathbb{R}^{m_{i}\times m_{i}}$ corresponding to eigenvalue $|\lambda_{i}|<1$, when $k\gg m_{i}$, one obtains
    \[
    [J_{i}^k]_{u,v}=\left\{\begin{array}{l}
     \lambda_{i}^k, \text{if} \ u=v;\\
     \binom{k}{v-u}\lambda_{i}^{k-{(v-u)}}, \text{if} \ u<v\leq m_{i};\\
     0, \text{otherwise}.
       \end{array}\right.
       \]
    It follows that the convergence rate of $J_{i}^{k}$ as $k\rightarrow \infty$ is governed by $\rho(J_{i})$. Thus, we consider $-\text{log}(\rho(J_{i}))$ as the convergence rate of $J_{i}^{k}$.

For one Jordan block matrix $J_{i}\in\mathbb{R}^{m_{i}\times m_{i}}$ corresponding to semisimple eigenvalue $\lambda_{i}=1$, when $k\gg m_{i}$, one obtains
  \[
  [J_{i}^k]_{u,v}=\left\{\begin{array}{l}
   \lambda_{i}^k=1, \ \text{if} \ u=v;\\
   0, \ \text{if} \ u\neq v.
     \end{array}\right.
     \]
   Then, we find that the convergence rate of  $W^{k}$ depends on how quickly $\rho(J_{i})$ goes to zero. Consequently, we can generalize the convergence rate of opinions from the unsigned network, as discussed in \cite{sen}, to the signed network scenario. Specifically, the convergence rate $r$ of system (\ref{equ001}) is
\[
\left\{\begin{array}{l}
   \text{-log}(\rho(W)), \text{if} \ \rho(W)<1,\\
   \text{-log}(\rho_{2}(W)), \text{otherwise}.
   \end{array}\right.
   \]
\end{proof}

From Lemma \ref{lem7}, according to the structural balance theory, we can identify the convergence rate of system (\ref{equ001}) by the spectral radius or the second-largest modulus eigenvalue of $W$. However, it remains unclear whether complex interaction types affect the convergence speed of the system. Next, we provide an example to demonstrate the significant impact that various interaction types have on the convergence rate $r$.

\begin{figure*}[ht]
    \centering
    \includegraphics[scale=0.32]{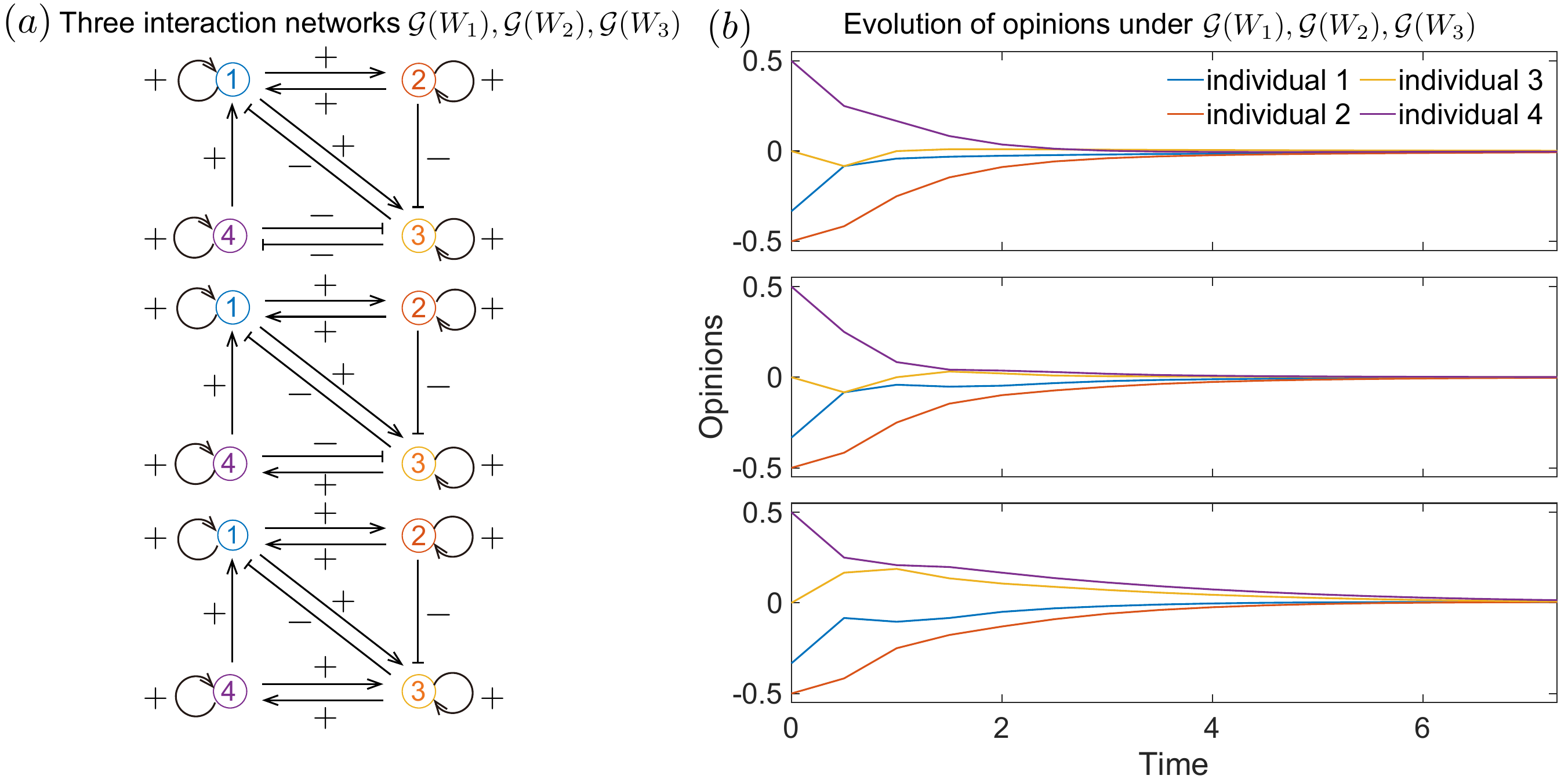}
    \caption{The effect of different interaction types on the convergence rate of system (\ref{equ001}). As shown in Fig. \ref{pic-0} (a), the interaction type between individuals 3 and 4 varies, whereas the other interaction types remain consistent across those networks.}
\label{pic-0}
  \end{figure*}
\begin{example}\label{exa1}\textup{
Consider the convergence rate of system (\ref{equ001}) on three interaction networks $\mathcal{G}(W_{1})$, $\mathcal{G}(W_{2})$, $\mathcal{G}(W_{3})$, where the influence matrices corresponding to these three networks are respectively given as follows: \begin{equation*}\begin{small}W_{1}=\begin{bmatrix}
0.25 & 0.25& -0.25& 0.25\\
0.5 & 0.5& 0& 0\\
0.25 & -0.25& 0.25& -0.25\\
0 & 0& -0.5& 0.5
\end{bmatrix}\end{small},\end{equation*} \begin{equation*}\begin{small}W_{2}=\begin{bmatrix}
0.25 & 0.25& -0.25& 0.25\\
0.5 & 0.5& 0& 0\\
0.25 & -0.25& 0.25& -0.25\\
0 & 0& 0.5& 0.5
\end{bmatrix}\end{small},\end{equation*} \begin{equation*}\begin{small}W_{3}=\begin{bmatrix}
0.25 & 0.25& -0.25& 0.25\\
0.5 & 0.5& 0& 0\\
0.25 & -0.25& 0.25& 0.25\\
0 & 0& 0.5& 0.5
\end{bmatrix}\end{small}.\end{equation*}}

\textup{Note that those three networks share identical structures and edge weights in value, but the signs (interaction types) associated with edges in three networks are distinct. For instance, $W_{1,43}=-W_{2,43}$. Moreover, $\rho(W_{1})=0.8536$, $\rho(W_{2})=0.7203$, and $\rho(W_{3})=0.7818$. As shown in  FIG. \ref{pic-0}, complex interaction types have a significant effect on the convergence rate of the opinion dynamics model, which is often neglected in previous literature.}
\end{example}
\section{Main results}\label{main1}
In this section, we first derive the convergence rate for random mixture interactions and further quantify the effect that some key factors have on the convergence rate, including population size $n$, network connectivity $P$, and self-confidence level $d$. After that, the convergence rate of system (\ref{equ001}) with complex mixture interactions is discussed. Finally, we present the effect of mutual interactions on convergence rate by considering two mixture scenarios $(+/+,-/-)$ and $(-/-,-/0)$. It should be noted that the convergence rate is considered for large population size $n$ in this paper.
\subsection{Random mixture interactions}
\begin{theorem}\label{the1}\textup{
The convergence rate of the system (\ref{equ001}) with random mixture interactions is
 \begin{align*}-\text{log}\left(\frac{\sqrt{nP\sigma^2}+d}{(n-1)P\mathbb{E}(|Z|)+d}\right).\end{align*}}
\end{theorem}

\begin{proof}
 For random mixture interactions, we first consider the eigenvalue distribution of matrix $\bar{W}=W-\text{diag}\{d/(d+C_{\text{r}})\}$. Then, some statistics of matrix $\bar{W}$ are given as follows: \[
 \left\{\begin{array}{l}
\mathbb{E}\left(\bar{W}_{ij}\right)=\displaystyle\frac{\mathbb{E}\left(S_{ij}\right)}{d+C_{\text{r}}}=0, \\
\mathbb{E}\left(\bar{W}_{ij}^{2}\right)=\displaystyle\frac{\mathbb{E}\left(S_{ij}^{2}\right)}{(d+C_{\text{r}})^2}=\displaystyle\frac{P\sigma^2}{(d+C_{\text{r}})^2}, \\%=\frac{\sigma^2}{P(n-1)^2\mathbb{E}^{2}(|Z|)}
\operatorname{Var}\left(\bar{W}_{i j}\right)=\mathbb{E}\left(\bar{W}_{ij}^{2}\right)-\mathbb{E}^{2}\left(\bar{W}_{ij}\right)=\mathbb{E}\left(\bar{W}_{ij}^{2}\right).%=\frac{\sigma^2}{P(n-1)^2\mathbb{E}^{2}(|Z|)}
\end{array}\right.\]
Let $F=\displaystyle\frac{\bar{W}}{\sqrt{n\operatorname{Var}\left(\bar{W}_{i j}\right)}}$, then we have
\[
\left\{\begin{array}{l}
\mathbb{E}\left(F_{ij}\right)=0, \\
\mathbb{E}\left(F_{ij}^{2}\right)=\displaystyle\frac{1}{n}, \\
\operatorname{Var}\left(F_{i j}\right)=\mathbb{E}\left(F_{ij}^{2}\right)-\mathbb{E}^{2}\left(F_{ij}\right)=\displaystyle\frac{1}{n}.
\end{array}\right.
\]
 According to Lemma \ref{lem3}, the eigenvalues of $F$ are uniformly distributed in a unit circle centered at $(0,0)$, as $n \rightarrow \infty$. It follows that when $n$ is sufficiently large, the eigenvalue distribution of $\bar{W}$ is uniform distributed in a circle of radius approximately

\vspace{-0.5cm}
 \begin{align*}\rho(\bar{W})=\sqrt{n\operatorname{Var}\left(\bar{W}_{i j}\right)}=\displaystyle\frac{\sqrt{nP\sigma^2}}{C_{\text{r}}+d}=\frac{\sqrt{nP\sigma^2}}{(n-1)P\mathbb{E}(|Z|)+d}.\end{align*}

Finally, note that the effect of $W_{ii}$: this shifts the circle so that it is now centered at $(W_{ii},0)$. Then, we have
\[
\rho(W)=\frac{\sqrt{nP\sigma^2}+d}{(n-1)P\mathbb{E}(|Z|)+d}.
\]
When $n$ is large, we have \begin{equation*}(n-1)P\mathbb{E}(|Z|)+d \approx nP\mathbb{E}(|Z|)+d >\sqrt{nP\sigma^2}+d,\end{equation*} i.e., $\rho(W)<1.$ By Lemma \ref{lem7}, the convergence rate is given as

\vspace{-0.5cm}
\begin{align}r&=-\text{log}(\rho(W))=-\text{log}\left(\frac{\sqrt{nP\sigma^2}+d}{(n-1)P\mathbb{E}(|Z|)+d}\right)\label{equ_32}.\end{align}
\end{proof}
\begin{remark}
\textup{According to Theorem \ref{the1}, for random mixture interactions, it is shown that $\rho(W)<1$ holds from an algebraic perspective, i.e., the system (\ref{equ001}) is stable. In fact, from the viewpoint of signed graph theory, we can also analyze the reasons for stability via Theorem 1 in \cite{Xia:16}. Based on existing results, the following equivalent results hold: $\rho(W)<1$ $\Leftrightarrow$ all CSCCs are structurally unbalanced $\Leftrightarrow$ there is at least one negative cycle in each CSCC. When $n$ is large, the random mixture interaction network is strongly connected and there exists at least one negative cycle. In what follows, Example \ref{exa2} is given to further illustrate this point.}
\end{remark}
\begin{example}\label{exa2}
\textup{For the random mixture of different interaction types, since $Z\sim\mathcal{N}(0,\sigma^2)$, when $n$ is large enough, we have $P_{+/+}=P_{-/-}=0.25$, $P_{+/-}=0.5$, and $P_{+/0}=P_{-/0}=0$. Consider the probability of of the random interaction network $\mathcal{G}(W)$ with three individuals being structurally unbalanced. }

\textup{By the definition of structural balance, the signed interaction network is structurally balanced if and only if the interaction scenario is just $(+/+,+/+,+/+)$ or $(-/-,-/-,+/+)$ (see FIG. \ref{pic-fs}). By some calculations, the probability of $\mathcal{G}(W)$ being structurally unbalanced is $0.8$. As the population size increases, the probability that $\mathcal{G}(W)$ is structurally balanced tends to 1, which implies that $\rho(W)<1$ for large $n$.}
\begin{figure}[ht]
    \centering
    \includegraphics[scale=0.35]{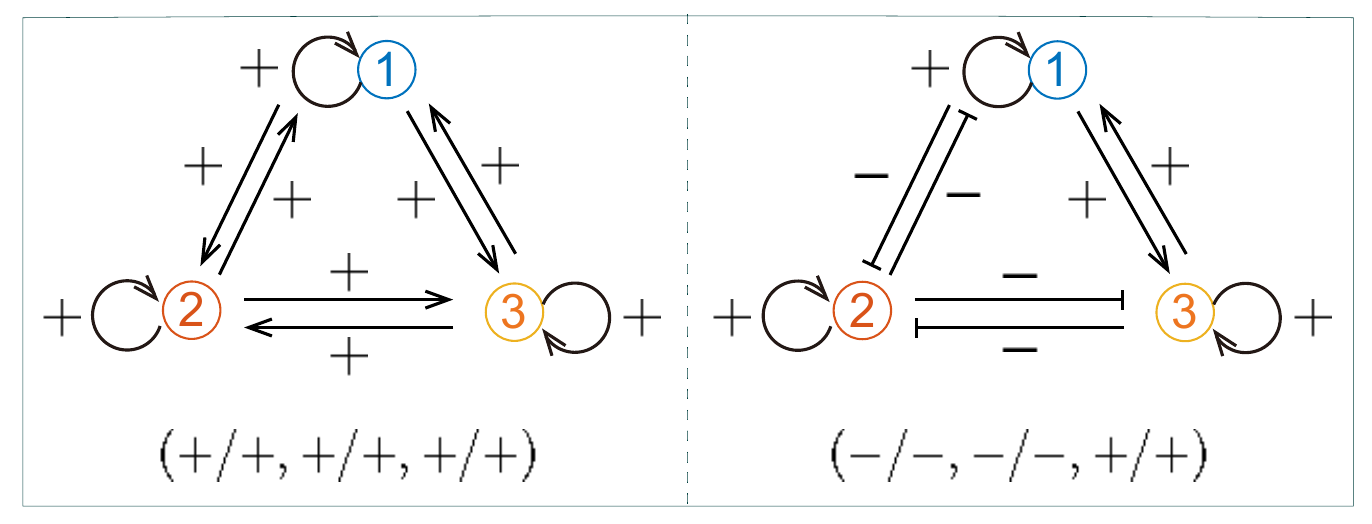}
    \caption{The signed interaction network $\mathcal{G}(W)$ is structurally balanced.}
    \label{pic-fs}
    \end{figure}
\end{example}
\begin{corollary}\label{cor1}\textup{
 For the system (\ref{equ001}) with random mixture interactions, the convergence rate is proportional to the population size and network connectivity, while it is inversely proportional to the individuals' self-confidence level.}
\end{corollary}
%lead to faster convergence when individuals interact randomly in a large-scale population. However, a higher self-confidence level results in a slower convergence rate.
\indent
\begin{proof} We prove this corollary in the following three steps:
\subsubsection{The effect of population size $n$ on convergence rate}

Differentiating (\ref{equ_32}) with respect to $n$, one obtains

\vspace{-0.5cm}
\begin{align*}&\frac{\partial (\rho(W))}{\partial n}=\frac{\frac{\eta\sqrt{P\sigma^2}}{2\sqrt{n}}-P\mathbb{E}(|Z|)(\sqrt{nP\sigma^2}+d)}{\eta^2},
\\ &=\frac{\frac{\eta\sqrt{P\sigma^2}}{2}-P\mathbb{E}(|Z|)(n\sqrt{P\sigma^2}+\sqrt{n}d)}{\sqrt{n}\eta^2},
\\ &=\frac{-\frac{n+1}{2}P\mathbb{E}(|Z|)\sqrt{P\sigma^2}+(\frac{\sqrt{P\sigma^2}}{2}-P\mathbb{E}(|Z|)\sqrt{n})d}{\sqrt{n}\eta^2}.\end{align*}
where $\eta=(n-1)P\mathbb{E}(|Z|)+d$. Since for lager $n$, \[
\frac{\sqrt{P\sigma^2}}{2}-P\mathbb{E}(|Z|)\sqrt{n}<0,\] then $\frac{\partial (\rho(W))}{\partial n}<0$, i.e., larger population size improves the convergence rate of system (\ref{equ001}).
\subsubsection{The effect of network connectivity $P$ on convergence rate}

Differentiating (\ref{equ_32}) with respect to $n$, one obtains

\vspace{-0.5cm}
\begin{align*}&\frac{\partial (\rho(W))}{\partial n}=\frac{\frac{\eta\sqrt{n\sigma^2}}{2\sqrt{P}}-(n-1)\mathbb{E}(|Z|)(\sqrt{nP\sigma^2}+d)}{\eta^2},
\\ &=\frac{-\frac{n-1}{2}P\mathbb{E}(|Z|)\sqrt{n\sigma^2}+(\frac{\sqrt{n\sigma^2}}{2}-(n-1)\mathbb{E}(|Z|)\sqrt{P})d}{\sqrt{n}\eta^2}.\end{align*}
Since for lager $n$, \[
\frac{\sqrt{n\sigma^2}}{2}-(n-1)\mathbb{E}(|Z|)\sqrt{P}<0,\] then $\frac{\partial (\rho(W))}{\partial P}<0$, i.e., larger network connectivity leads to faster convergence of system (\ref{equ001}).
%Strong network connectivity enables sufficient and quick interaction, which means larger network connectivity can improve the spread of information in society (see Fig. \ref{fig8b}(b)).
%by Theorem \ref{the1}, for fixed $n$, the convergence rate can be estimated as: \begin{align*}r_d&=-\text{log}\left(\frac{\sqrt{nP\sigma^2}+d}{(n-1)P\mathbb{E}(|Z|)+d}\right)
%\\&\approx-\text{log}\left(\frac{\sqrt{nP\sigma^2}}{(n-1)P\mathbb{E}(|Z|)}\right)
%\\&\approx-\text{log}\left(\frac{\sqrt{n\sigma^2}}{(n-1)\sqrt{P}\mathbb{E}(|Z|)}\right).\end{align*}
\subsubsection{The effect of self-confidence level $d$ on convergence rate}

Differentiating (\ref{equ_32}) with respect to $d$, we obtain

\vspace{-0.5cm}
\begin{align*}\frac{\partial (\rho(W))}{\partial d}&=\frac{(n-1)P\mathbb{E}(|Z|)-\sqrt{nP}\sigma}{\eta^2},
\\&\approx\frac{\sqrt{nP}(nP\mathbb{E}(|Z|)-\sigma)}{\eta^2}
>0.\end{align*}
Therefore, higher self-confidence level can decrease the convergence rate of system (\ref{equ001}).
%In society, if all individuals have very high self-confidence levels, they will be constrained by their own opinions over time, which means system (\ref{equ001}) achieves convergence slowly (see Fig. \ref{fig8b}(c)).
\end{proof}
 % \begin{figure}
 % \centering
 %    \includegraphics[scale=0.4]{zixin.pdf}
 %    \caption{Estimating the convergence rate with different self-confidence $W_{ii}$. The other parameters are the same with that in Fig. \ref{pic-random}. Each dot is an average of 50 rates with the same set of parameters. As the self-confidence increases, the convergence rate monotonically decreases.}
  %   \label{pic-zixin}
  %  \end{figure}
\begin{remark}
In most existing literature, an upper bound of the convergence rate that implicitly incorporates parameters of network properties is obtained \cite{Ols,Xu, Liu17}. However, it is difficult to directly identify the key factors that have significant impacts on the convergence rate. In contrast to previous studies, Theorem \ref{the1} provides a new perspective derived from the analysis of random matrices and explicitly establishes a specific expression for the convergence rate. Moreover, Corollary \ref{cor1} quantifies the effect that some key factors have on it, thereby overcoming some limitations of existing research methods.
\end{remark}
\subsection{Complex mixture interactions}
In this subsection, we present the convergence rate of system $(\ref{equ001})$ with complex mixture interactions and study the effect of different interaction types on it. In Theorem \ref{the2}, we do not consider two trivial cases, where all interaction types are mistrust or trust. Hence, we need to give the following assumption.

\begin{assumption}\label{ass1}\textup{
There coexist trust and mistrust interactions in the interaction network $\mathcal{G}(W)$, i.e., $P_{+/+}+P_{+/0}\neq 1$ and $P_{-/-}+P_{-/0}\neq 1$.}
\end{assumption}

From equations (\ref{equ_28}) and (\ref{equ_29}), we can obtain Lemma \ref{lem51} below.

\begin{lemma}\label{lem51} \textup{Assumption \ref{ass1} holds if and only if $\hat{P}\neq\bar{P}$ and $\hat{P}\neq-\bar{P}$.}
  \end{lemma}
  \begin{proof} We prove this lemma by contradiction.

 ($\Rightarrow$) Suppose $\bar{P}=\hat{P}$, by simple calculations, then one obtains %$$P_{+/+}+P_{-/-}+P_{+/-}+\frac{1}{2}P_{+/0}+\frac{1}{2}P_{+/0}=P_{+/+}-P_{-/-}+\frac{1}{2}P_{+/0}-\frac{1}{2}P_{/0},$$ i.e.,
\begin{equation}\label{equ_9}P_{+/-}=-2P_{-/-}-P_{-}.\end{equation}
Substituting (\ref{equ_9}) into $P_{+/+}+P_{-/-}+P_{+/-}+P_{+/0}+P_{-/0}=1$, we have
\begin{equation*}P_{+/+}-P_{-/-}+P_{+/0}=1.\end{equation*} Since $P_{-/-}\geq0$, we have \begin{equation*}P_{+/+}+P_{+/0}=1,\end{equation*} which means that there only exist trust interactions. Similarly, we can also prove that if $\bar{P}=-\hat{P}$, then $P_{-/-}+P_{-/0}=1$. Since this contradicts Assumption \ref{ass1}, the sufficiency holds.

($\Leftarrow$) Suppose $P_{+/+}+P_{+/0}=1$, then \begin{equation}\label{equ_27}P_{+/-}=P_{-/-}=P_{-/0}=0.\end{equation}
Substituting (\ref{equ_27}) into (\ref{equ_28}) and (\ref{equ_29}) respectively, we have
\[
\hat{P}=P_{+/+}+\frac{1}{2}P_{+/0}=\bar{P}.
\]
By similar analysis, we can prove the sufficiency holds for $P_{-/-}+P_{-/0}=1$ case. Since this contradicts the premise $P_{+/+}+P_{+/0}\neq 1$ and $P_{-/-}+P_{-/0}\neq 1$, the necessity is established. In summary, we complete the proof of this lemma.
  \end{proof}

\begin{theorem}\label{the2} \textup{Under Assumption \ref{ass1}, the convergence rate $r$ of system (\ref{equ001}) with complex mixture interactions is
\[
\left\{\begin{array}{l}
   -\text{log}(M_{e}), \text{if} \ |n\mathbb{E}|\leq\sqrt{n\mathbb{V}},\\
   -\text{log}(\max(|\lambda_{\text{outlier}}|,M_{e})), \text{otherwise},
   \end{array}\right.
   \]
where $M_{e}=\max(\Delta_{1}, \Delta_{2})$ and
\[
\left\{\begin{array}{l}
\Delta_{1}=\sqrt{n\mathbb{V}}(1+\frac{\mathbb{T}-\mathbb{E}^2}{\mathbb{V}})+|-\mathbb{E}+d|,\\
%\Delta_{2}=-\sqrt{n\mathbb{V}}(1+\frac{\mathbb{T}-\mathbb{E}^2}{\mathbb{V}})+|-\mathbb{E}+W_{ii}|,\\
\Delta_{2}=\sqrt{n\mathbb{V}(1-\frac{\mathbb{T}-\mathbb{E}^2}{\mathbb{V}})^2+(-\mathbb{E}+d)^2},\\
\end{array}\right.
\]
 $\lambda_{\text{outlier}}=(n-1)\mathbb{E}+\frac{\mathbb{T}-\mathbb{E}^2}{\mathbb{E}}+W_{ii}$. $\mathbb{V}$, $\mathbb{T}$ and $\mathbb{E}$ are defined in (\ref{equ_33}) below.}
\end{theorem}
\begin{proof}
Just as for random mixture interactions, we first consider the eigenvalue distribution of matrix $\bar{W}=W-\text{diag}\{W_{ii}\}$. Then, we have
 \begin{equation}\label{equ_33}\left\{\begin{array}{l}
\mathbb{E}:=\mathbb{E}\left(\bar{W}_{ij}\right)=\frac{\mathbb{E}\left(S_{ij}\right)}{C_{\text{m}}+d}=\frac{P\bar{P}\mathbb{E}\left(|Z|\right)}{C_{\text{m}}+d},\\%=\displaystyle\frac{\bar{P}}{(n-1)\hat{P}}, \\
\mathbb{E}\left(\bar{W}_{ij}^{2}\right)=\frac{\mathbb{E}\left(S_{ij}^{2}\right)}{(C_{\text{m}}+d)^2}=\frac{P\hat{P}\sigma^2}{(C_{\text{m}}+d)^2},\\%=\displaystyle \frac{\sigma^2}{P\hat{P}(n-1)^2\mathbb{E}^{2}(|Z|)}, \\
\mathbb{V}:=\operatorname{Var}\left(\bar{W}_{i j}\right)=\mathbb{E}\left(\bar{W}_{ij}^{2}\right)-\mathbb{E}^2,\\%=\displaystyle \frac{\hat{P}\sigma^2-P\bar{P}^{2}\mathbb{E}^{2}(|Z|)}{P\hat{P}^{2}(n-1)^2\mathbb{E}^{2}(|Z|)},\\
\mathbb{T}:=\mathbb{E}\left(\bar{W}_{ij}\bar{W}_{ji}\right)=\frac{PP^{*}\mathbb{E}^{2}\left(|Z|\right)}{(C_{\text{m}}+d)^2},%=\displaystyle \frac{P_{+/+}+P_{mm}-P_{tm}}{P\hat{P}^{2}(n-1)^{2}}.
\end{array}\right. \end{equation}
where $\hat{P}$, $\bar{P}$ and $P^{*}$ are defined in (\ref{equ_28}), (\ref{equ_29}) and (\ref{equ_299}), respectively.
%Furthermore, since all diagonal entries $W_{ii}=\frac{d}{C_{\text{m}}+d}>0$, the eigenvalue distribution of $\bar{W}$ is shifted leftwards $\frac{d}{C_{\text{m}}+d}$ along the horizontal axis.

Let $N=\bar{W}-\mathbb{E} \cdot \textbf{1}\cdot \textbf{1}^{T}+\mathbb{E}\cdot I$, then we can obtain some statistics of matrix $N$. Specifically,
\[
\left\{\begin{array}{l}
\mathbb{E}\left(N_{ij}\right)=0,\\
\mathbb{E}\left(N_{ij}^{2}\right)=\mathbb{V}=\operatorname{Var}\left(N_{i j}\right),\\%=\displaystyle \frac{\hat{P}\sigma^2-P\bar{P}^{2}\mathbb{E}^{2}(|Z|)}{P\hat{P}^{2}(n-1)^2\mathbb{E}^{2}(|Z|)}, \\
\mathbb{E}\left(N_{ij}N_{ji}\right)=\mathbb{E}\left((\bar{W}_{ij}-\mathbb{E} )(\bar{W}_{ji}-\mathbb{E} )\right)=\mathbb{T}-\mathbb{E}^2.%=\displaystyle \frac{P_{+/+}+P_{mm}-P_{tm}-P\bar{P}^{2}}{P\hat{P}^{2}(n-1)^{2}}.
\end{array}\right.
\]

In the sequel, let $F=N/\sqrt{n\mathbb{V}}$, then

\[\left\{\begin{array}{l}
\mathbb{E}\left(F_{ij}\right)=0, \\
\mathbb{E}\left(F_{ij}^{2}\right)=\displaystyle\frac{1}{n}, \\
\mathbb{E}\left(F_{ij}F_{ji}\right)=\displaystyle\frac{\tau}{n}, \\
\end{array}\right.
\]
where $\tau=(\mathbb{T}-\mathbb{E}^2)/\mathbb{V}$. According to Lemma \ref{lem4}, when $n$ is sufficiently large, the eigenvalues of $F$ are uniformly distributed in an ellipse centered at (0, 0) and

\[
{\left(\frac{x}{1+\tau}\right)}^{2}+{\left(\frac{y}{1-\tau}\right)}^{2}\leq 1.
\] It follows that the eigenvalues of $N$ are uniformly distributed in an ellipse centered at (0, 0) and

\[
{\left(\frac{x}{\sqrt{n\mathbb{V}}(1+\tau)}\right)}^{2}+{\left(\frac{y}{\sqrt{n\mathbb{V}}(1-\tau)}\right)}^{2}\leq 1.
\]
%$${\left(\frac{x}{\left(1+\tau\right)\sqrt{n\operatorname{Var}\left(N_{i j}\right)}}\right)}^{2}+\left({\frac{x}{\left(1-\tau\right)\sqrt{n\operatorname{Var}\left(N_{i j}\right)}}\right)}^{2}\leq 1$$

Notice $\mathbb{E} \cdot \textbf{1}\cdot \textbf{1}^{T}$ is a rank-one perturbation matrix with $n-1$ zero eigenvalues and $n\mathbb{E}$ is a single eigenvalue. According to the low-rank perturbation theorem, when $|n\mathbb{E}|\leq \sqrt{n\mathbb{V}}$, all eigenvalues of $N+\mathbb{E} \cdot \textbf{1}\cdot \textbf{1}^{T}$  are still uniformly distributed in the ellipse above. When $|n\mathbb{E}|>\sqrt{n\mathbb{V}},$ $n-1$ eigenvalues of $N+\mathbb{E} \cdot \textbf{1}\cdot \textbf{1}^{T}$ are still uniformly distributed in the ellipse above, whereas an eigenvalue $\hat\lambda$ is modified as

\vspace{-0.5cm}
 \begin{align*}\hat\lambda&=n\mathbb{E}\left(\bar{W}_{ij}\right)+\frac{\mathbb{E}\left(N_{ij}N_{ji}\right)}{\mathbb{E}\left(\bar{W}_{ij}\right)}
=n\mathbb{E}+\frac{\mathbb{T}-\mathbb{E}^2}{\mathbb{E}}.\end{align*}%\\&=\frac{(n-1)P\bar{P}^{2}+P_{+/+}+P_{mm}-P_{tm}}{(n-1)P\bar{P}\hat{P}}
Since all diagonal entries $W_{ii}=\frac{d}{C_{\text{m}}+d}>0$, the eigenvalue distribution of $W$ is shifted leftwards $\frac{d}{C_{\text{m}}+d}$ along the horizontal axis. Therefore, for sufficiently large $n$, we obtain the eigenvalue distribution of $W$

%\approx\frac{\bar{P}}{\hat{P}}
When $|n\mathbb{E}|\leq \sqrt{n\mathbb{V}}$, the eigenvalues of $W$  are uniformly distributed in the ellipse
\begin{equation}\label{equ42}{\left(\frac{x+\mathbb{E}-W_{ii}}{a_{N}}\right)}^{2}+{\left(\frac{y}{b_{N}}\right)}^{2}\leq 1,\end{equation}
where $a_{N}=\sqrt{n\mathbb{V}}(1+\tau)$ and $b_{N}=\sqrt{n\mathbb{V}}(1-\tau)$.

When $|n\mathbb{E}|>\sqrt{n\mathbb{V}}$, there is also an eigenvalue distributed outside the ellipse
\begin{equation}\label{equ43}\left\{\begin{array}{l}
\displaystyle {\left(\frac{x+\mathbb{E}-W_{ii}}{a_{N}}\right)}^{2}+{\left(\frac{y}{b_{N}}\right)}^{2}\leq 1,\\
\displaystyle \lambda_{\text{outlier}}=\hat{\lambda}-\mathbb{E}+W_{ii}.
\end{array}\right.\end{equation}

%When $n$ is sufficiently large, since the translation of an ellipse on the horizontal axis $\mathbb{E}:=\mathbb{E}\left(S_{ij}\right)/(nP\hat{P}\mathbb{E}\left(|Z|\right))=\bar{P}/(n\hat{P})$ is approximately equal to 0,$Q_{\text{rightmost}}$, $Q_{\text{leftmost}}$ and $Q_{\text{uppermost}}$ $Q_{\text{outlier}}$
%=(a_{N}+\frac{d-P\bar{P}\mathbb{E}\left(|Z|\right)}{d+(n-1)P\hat{P}\mathbb{E}\left(|Z|\right)}, 0)
 Three endpoints $Q_{\text{rightmost}}$, $Q_{\text{leftmost}}$ and $Q_{\text{uppermost}}$ (rightmost, leftmost, and uppermost) of the distribution in above ellipse and the unique endpoint $Q_{\text{outlier}}$ corresponding to eigenvalue $\lambda_{\text{outlier}}$ outside the ellipse (if it exists) can then be estimated as
\[
\left\{\begin{array}{l}
Q_{\text{rightmost}}=(a_{N}-\mathbb{E}+W_{ii}, 0),\\
Q_{\text{leftmost}}=(-a_{N}-\mathbb{E}+W_{ii}, 0), \\
Q_{\text{uppermost}}=(-\mathbb{E}+W_{ii}, b_{N}),\\
Q_{\text{outlier}}=(\lambda_{\text{outlier}}, 0)=(\hat{\lambda}-\mathbb{E}+W_{ii}, 0).
\end{array}\right.
\]
%where a_{W}=a_{N}-\mathbb{E}+\frac{1}{n}
Furthermore,
\begin{equation}\label{equ_12}\left\{\begin{array}{l}
\mathbb{E}=\frac{\mathbb{E}\left(S_{ij}\right)}{C_{\text{m}}+d}=\frac{P\bar{P}\mathbb{E}\left(|Z|\right)}{(n-1)P\hat{P}\mathbb{E}\left(|Z|\right)+d}, \\
W_{ii}-\mathbb{E}=\frac{d-P\bar{P}\mathbb{E}\left(|Z|\right)}{(n-1)P\hat{P}\mathbb{E}\left(|Z|\right)+d}\\
\sqrt{n\mathbb{V}}=\frac{\sqrt{nP\hat{P}\sigma^2-nP^2\bar{P}^2\mathbb{E}^2\left(|Z|\right)}}{(n-1)P\hat{P}\mathbb{E}\left(|Z|\right)+d},\\
\tau=\frac{P^{*}\mathbb{E}^2\left(|Z|\right)-P\bar{P}^2\mathbb{E}^2\left(|Z|\right)}{\hat{P}\sigma^2-P\bar{P}^2\mathbb{E}^2\left(|Z|\right)},\\
\lambda_{\text{outlier}}=\frac{P^{*}\mathbb{E}^2\left(|Z|\right)+(n-2)P\bar{P}^2\mathbb{E}^2\left(|Z|\right)+d\bar{P}\mathbb{E}\left(|Z|\right)}{\bar{P}\mathbb{E}\left(|Z|\right)((n-1)P\hat{P}\mathbb{E}\left(|Z|\right)+d)}.
\end{array}\right.\end{equation}

According to (\ref{equ_12}), for large $n$, $\sqrt{n\mathbb{V}}>0$ is sufficiently small and $\tau$ is bound. Moreover, $a_{N}$ and $b_{N}$ are sufficiently small. Furthermore, $|W_{ii}-\mathbb{E}|<1$. Hence, $M_{e}=\max\{\Delta_{1},\Delta_{2}\}<1$. Moreover,
for large $n$, we have

\vspace{-0.5cm}
\begin{align*}\lambda_{\text{outlier}}&\approx\frac{(n-2)P\bar{P}^2\mathbb{E}^2\left(|Z|\right)
+d\bar{P}\mathbb{E}\left(|Z|\right)}{\bar{P}\mathbb{E}\left(|Z|\right)((n-1)P\hat{P}\mathbb{E}\left(|Z|\right)+d)}\nonumber\\
&\approx\frac{(n-1)P\bar{P}\mathbb{E}^2\left(|Z|\right)
+d\mathbb{E}\left(|Z|\right)}{(n-1)P\hat{P}\mathbb{E}^2\left(|Z|\right))+d\mathbb{E}\left(|Z|\right)}\label{equ_34}.\end{align*}
Thus, $|\lambda_{\text{outlier}}|\approx 1$ if and only if $|\bar{P}|=\hat{P}$. By Assumption \ref{ass1} and Lemma \ref{lem51}, the coexistence of trust and mistrust interactions ensures $\bar{P}\neq\hat{P}$, i.e., $|\lambda_{\text{outlier}}|<1$. Thus, the convergence rate $r$ of system (\ref{equ001}) is
\[
\left\{\begin{array}{l}
   -\text{log}(M_{e}), \text{if} \ |n\mathbb{E}|\leq\sqrt{n\mathbb{V}},\\
   -\text{log}(\max(|\lambda_{\text{outlier}}|,M_{e})), \text{otherwise}.
   \end{array}\right.
   \]
\end{proof}

In Theorem \ref{the2}, we have established the quantitative expression of the convergence rate with respect to certain key factors, providing a theoretical foundation for our further analysis of how these factors influence the speed of opinion evolution. In the following subsections, we will present more specific findings in some typical scenarios, including both pure and mixed interactions, which will help us understand the role of complex interaction types in the process of opinion evolution.

\subsection{Convergence rate $r$ for five typical interaction types $(+/+),(-/-),(+/-),(+/0)$, and $(-/0)$}
Based on the signs of $W_{ij}/W_{ji}$, there are five typical types of interactions, namely, mutual trust $(+/+)$, mutual mistrust $(-/-)$, trust$/$mistrust  $(+/-)$, unilateral trust  $(+/0)$, unilateral mistrust $(-/0)$. In the sequel, we will explore the convergence rate of system (\ref{equ001}) with these five interaction types.
  \begin{assumption}\label{ass2}
Suppose $\underline{d}<d\ll n$, where $n$ is large and $\underline{d}=\max\{M_{1}, M_{2}\},$ where \[
M_{1}=\frac{(\sigma^2+\mathbb{E}^2\left(|Z|\right))^2}{\sigma^2\mathbb{E}\left(|Z|\right)}\] and
\[
M_{2}=\frac{2\sigma^4}{\sqrt{2\sigma^2\mathbb{E}^2\left(|Z|\right)-P\mathbb{E}^4\left(|Z|\right)}}+\frac{P\mathbb{E}\left(|Z|\right)}{2}.\]
\end{assumption}

\begin{remark}
 In the real world, an individual's self-confidence level is usually not infinitely low. From a psychological perspective, it is shown that individuals' self-confidence level usually do not drop to zero because people maintain a certain level of self-esteem and self-worth even when faced with failure. Assumption \ref{ass2} implies that people maintain at least a certain level of self-confidence, which is helpful for us to facilitate our theoretical analysis in Theorem \ref{the3}.
\end{remark}
\begin{theorem}\label{the3}\textup{
Under Assumption \ref{ass2}, the following statements hold:}

\noindent
\textup{(i). For the mutual trust $(+/+)$, trust$/$mistrust $(+/-)$, and unilateral trust $(+/0)$ scenarios, the convergence rate of system (\ref{equ001}) is proportional to the population size and network connectivity, while it is inversely proportional to the individuals' self-confidence level.}

\noindent
\textup{(ii). For the mutual mistrust $(-/-)$ and unilateral mistrust $(-/0)$, the conclusions are contrary to those presented in (i).}
\end{theorem}
\begin{proof}
\textbf{(1). Mutual trust $(+/+)$ interaction.}

In this case, $W$ is a row-stochastic matrix with 1 as its dominant eigenvalue, then $\rho(W)=1$.  By Lemmas \ref{lem4} and \ref{lem7}, the second-largest modulus eigenvalue of $W$ distributes in an ellipse (\ref{equ42}) and determines the convergence rate. Hence, $r=-\text{log}(M_{e})$.

Substituting $P_{+/+}=1$, $P_{-/-}=P_{+/-}=P_{+/0}=P_{-/0}=0$ into (\ref{equ_12}), we obtain
\[
\left\{\begin{array}{l}
\mathbb{E}=\frac{P\mathbb{E}\left(|Z|\right)}{(n-1)P\mathbb{E}\left(|Z|\right)+d}, \\
W_{ii}-\mathbb{E}=\frac{d-P\mathbb{E}\left(|Z|\right)}{(n-1)P\mathbb{E}\left(|Z|\right)+d}, \\
\sqrt{n\mathbb{V}}=\frac{\sqrt{nP}\sqrt{\sigma^2-P\mathbb{E}^2\left(|Z|\right)}}{(n-1)P\mathbb{E}\left(|Z|\right)+d},\\
\tau=\frac{\mathbb{E}^2\left(|Z|\right)-P\mathbb{E}^2\left(|Z|\right)}{\sigma^2-P\mathbb{E}^2\left(|Z|\right)}.
\end{array}\right.
\]
Furthermore, since $\tau>0$, by Lemma \ref{lem5}, \[
r=-\text{log}(\sqrt{n\mathbb{V}}(1+\tau)+|W_{ii}-\mathbb{E}|).\] Moreover, we have

\vspace{-0.5cm}
\begin{align}&\sqrt{n\mathbb{V}}(1+\tau)+|W_{ii}-\mathbb{E}|
\nonumber\\=&\frac{\sqrt{nP}(\alpha^2+(1-P)\mathbb{E}^2(|Z|))+\alpha|d-P\mathbb{E}(|Z|)|}{\alpha(n-1)P\mathbb{E}\left(|Z|\right)+\alpha d}\nonumber\\
\approx &\frac{\sqrt{nP}(\alpha^2+(1-P)\mathbb{E}^2(|Z|))}{\alpha nP\mathbb{E}\left(|Z|\right)}\nonumber \\ \approx & \frac{\alpha+\frac{1-P}{\alpha}\mathbb{E}^2(|Z|)}{\sqrt{nP}\mathbb{E}\left(|Z|\right)}\label{equ_4},\end{align}
where $\alpha=\sqrt{\sigma^2-P\mathbb{E}^2\left(|Z|\right)}$ and $d-P\mathbb{E}\left(|Z|\right)>0$ due to Assumption \ref{ass2}.
Hence, increasing population size accelerates convergence. Moreover, by differentiating $\frac{1-P}{\alpha}$ with respect to $P$, we obtain

\vspace{-0.5cm}
\begin{align*}\frac{\partial (\frac{1-P}{\alpha})}{\partial P}=\frac{\mathbb{E}^2(|Z|)-\sigma^2}{\alpha^{\frac{3}{2}}}<0.\end{align*}
Moreover, we can obtain that stronger network connectivity promotes convergence as well.

Let \[
H:=\frac{d-P\mathbb{E}(|Z|)+\beta}{(n-1)P\mathbb{E}(|Z|)+d},
 \]
 where $\beta=\frac{\sqrt{nP}(\sigma^2+(1-2P)\mathbb{E}^2(|Z|))}{\alpha}$. Then, we have

\vspace{-0.5cm}
\begin{align*}\frac{\partial H}{\partial d}=\frac{nP\mathbb{E}(|Z|)-\beta}{((n-1)P\mathbb{E}(|Z|)+d)^2}>0.\end{align*}
Thus, higher self-confidence level $d$ decreases convergence rate.
\textbf{(2). Trust$/$mistrust $(+/-)$ interaction.}

In this case, all eigenvalues distribute in the ellipse (\ref{equ42}), thus $r=-\text{log}(M_{e})$. Then, substituting $P_{+/-}=1$, $P_{+/+}=P_{-/-}=P_{+/0}=P_{-/0}=0$ into (\ref{equ_12}), we have
\[
\left\{\begin{array}{l}
\mathbb{E}=0,\\
W_{ii}-\mathbb{E}=\frac{d}{(n-1)P\mathbb{E}\left(|Z|\right)+d},\\
\sqrt{n\mathbb{V}}=\frac{\sigma\sqrt{nP}}{(n-1)P\mathbb{E}\left(|Z|\right)+d},\\
\tau=\frac{-\mathbb{E}^2\left(|Z|\right)}{\sigma^2}.
\end{array}\right.
\]
In the sequel, we first compute $\Delta_{1}$ and obtain%,
%\nonumber\\&\approx\frac{\sigma^2-\mathbb{E}^2\left(|Z|\right)}{\sqrt{nP}\mathbb{E}\left(|Z|\right)}\label{equ_23}

\vspace{-0.5cm}
\begin{align*}&\sqrt{n\mathbb{V}}(1+\tau)+|W_{ii}-\mathbb{E}|\\=&\frac{\sqrt{nP}(\sigma^2-\mathbb{E}^2\left(|Z|\right))+\sigma d}{\sigma((n-1)P\mathbb{E}\left(|Z|\right)+d)}\nonumber\\ \approx& \frac{\sqrt{nP}(\sigma^2-\mathbb{E}^2\left(|Z|\right))}{\sigma(n-1)P\mathbb{E}\left(|Z|\right)}\nonumber\\ \approx & \frac{\sigma^2-\mathbb{E}^2\left(|Z|\right)}{\sqrt{nP}\sigma\mathbb{E}\left(|Z|\right)},\end{align*}
and

\vspace{-0.5cm}
\begin{align*}&\frac{\partial (\sqrt{n\mathbb{V}}(1+\tau)+|W_{ii}-\mathbb{E}|)}{\partial d}\\=&\frac{(n-1)P\mathbb{E}\left(|Z|\right)\sigma^2-\sqrt{nP}\sigma(\sigma^2-\mathbb{E}^2\left(|Z|\right))+\sigma^2d-\sigma^2 d}{\sigma^2((n-1)P\mathbb{E}\left(|Z|\right)+d)^2}\\ \approx &\frac{nP\mathbb{E}\left(|Z|\right)\sigma^2-\sqrt{nP}\sigma(\sigma^2-\mathbb{E}^2\left(|Z|\right))}{\sigma^2((n-1)P\mathbb{E}\left(|Z|\right)+d)^2},
\\ = &\frac{\sqrt{nP}(\sqrt{nP}\mathbb{E}\left(|Z|\right)\sigma-\sigma^2+\mathbb{E}^2\left(|Z|\right))}
{\sigma((n-1)P\mathbb{E}\left(|Z|\right)+d)^2}>0.\end{align*}
Next, following the similar analysis for $\Delta_{1}$, we consider $\Delta_{2}$ and have

\vspace{-0.5cm}
\begin{align*}&n\mathbb{V}(1-\tau)^2+(W_{ii}-\mathbb{E})^2\\=&\frac{nP(\sigma^2+\mathbb{E}^2\left(|Z|\right))^2+\sigma^2 d^2}{\sigma^2((n-1)P\mathbb{E}\left(|Z|\right)+d)^2},
\nonumber\\ \approx &\frac{nP(\sigma^2+\mathbb{E}^2\left(|Z|\right))^2+\sigma^2 d^2}{\sigma^2(nP\mathbb{E}\left(|Z|\right)+d)^2}
\nonumber\\ \approx &\frac{(\sigma^2+\mathbb{E}^2\left(|Z|\right))^2}{nP\mathbb{E}^2\left(|Z|\right)\sigma^2},\end{align*}
and

\vspace{-0.5cm}
\begin{small}
\begin{align}&\frac{\partial (n\mathbb{V}(1-\tau)^2+(W_{ii}-\mathbb{E})^2)}{\partial d}\nonumber\\=&\frac{2d\sigma^2((n-1)P\mathbb{E}\left(|Z|\right)+d)-2(nP(\sigma^2+\mathbb{E}^2\left(|Z|\right))^2+\sigma^2 d^2)}{\sigma^2((n-1)P\mathbb{E}\left(|Z|\right)+d)^3}\nonumber\\ \approx &\frac{2nP\mathbb{E}\left(|Z|\right)d\sigma^2-2nP(\sigma^2+\mathbb{E}^2\left(|Z|\right))^2}{\sigma^2((n-1)P\mathbb{E}\left(|Z|\right)+d)^3}
\nonumber\\ = &\frac{2nP(\mathbb{E}\left(|Z|\right)d\sigma^2-(\sigma^2+\mathbb{E}^2\left(|Z|\right))^2)}
{\sigma^2((n-1)P\mathbb{E}\left(|Z|\right)+d)^3}>0,\label{equ44}\end{align}\end{small}
where the inequality (\ref{equ44}) holds according to Assumption \ref{ass2}. In summary, we can obtain the same conclusions as in case (1).

\textbf{(3). Unilateral trust $(+/0)$ interaction.}

In this case, $W$ is a row-stochastic matrix with 1 as its dominant eigenvalue and thereby the second-largest modulus eigenvalue of $W$ distributed in an ellipse determines the convergence rate. By Theorem \ref{the2}, $r=-\text{log}(M_{e})$. Then, substituting $P_{+/0}=1$, $P_{+/+}=P_{-/-}=P_{+/-}=P_{-/0}=0$ into (\ref{equ_12}), we have
\[
\left\{\begin{array}{l}
\mathbb{E}=\frac{\mathbb{E}\left(S_{ij}\right)}{d+C_{m}}=\frac{P\mathbb{E}\left(|Z|\right)}{2d+(n-1)P\mathbb{E}\left(|Z|\right)}, \\
W_{ii}-\mathbb{E}=\frac{2d-P\mathbb{E}\left(|Z|\right)}{2d+(n-1)P\mathbb{E}\left(|Z|\right)},\\
\sqrt{n\mathbb{V}}=\frac{\sqrt{2nP\sigma^2-nP^2\mathbb{E}^2\left(|Z|\right)}}{2d+(n-1)P\mathbb{E}\left(|Z|\right)},\\
\tau=\frac{-P\mathbb{E}^2\left(|Z|\right)}{2\sigma^2-P\mathbb{E}^2\left(|Z|\right)}.
\end{array}\right.
\]
Similarly, we first compute $\Delta_{1}$ and obtain

\vspace{-0.5cm}
\begin{align}&\sqrt{n\mathbb{V}}(1+\tau)+|W_{ii}-\mathbb{E}|
\nonumber\\=&\frac{\sqrt{nP}(2\sigma^2-2P\mathbb{E}^2\left(|Z|\right))
+\zeta_{1}\zeta_{2}}
{\zeta_{1}\zeta_{3}}
\nonumber\\ \approx&\frac{2\sigma^2-2P\mathbb{E}^2\left(|Z|\right)}{\zeta_{1}\sqrt{nP}\mathbb{E}\left(|Z|\right)
)}\label{equ_36},\end{align}
where \[\zeta_{1}=\sqrt{2\sigma^2-P\mathbb{E}^2\left(|Z|\right)},\] \[\zeta_{2}=2d-P\mathbb{E}\left(|Z|\right),\]
\[\zeta_{3}=2d+(n-1)P\mathbb{E}\left(|Z|\right).\]
%By differentiating $\frac{2\sigma^2-2P\mathbb{E}^2\left(|Z|\right)}{\zeta_{1}$ with respect to $P$:
Moreover, we have

\vspace{-0.5cm}
\begin{align}&\frac{\partial (\displaystyle\frac{2\sigma^2-2P\mathbb{E}^2\left(|Z|\right)}{\zeta_{1}})}{\partial P}=\frac{-6\mathbb{E}^2\left(|Z|\right)\sigma^2+2P\mathbb{E}^4\left(|Z|\right)}{\zeta_{1}^3}<0,\label{equ_37}
\end{align}
%= &\frac{2\zeta(2d+(n-1)P\mathbb{E}\left(|Z|\right))-2(\sqrt{nP}(2\sigma^2-2P\mathbb{E}^2\left(|Z|\right))
%+\zeta(2d-P\mathbb{E}\left(|Z|\right)))}
%{\zeta(2d+(n-1)P\mathbb{E}\left(|Z|\right))^2}\ref{equ36,equ37,equ39,equ40,equ41,equ45}
%\\
and
%\begin{footnotesize}

\vspace{-0.5cm}
\begin{align}&\frac{\partial (\sqrt{n\mathbb{V}}(1+\tau)+|W_{ii}-\mathbb{E}|)}{\partial d}\nonumber\\ \approx &\frac{2\zeta_{1} nP\mathbb{E}\left(|Z|\right)-\sqrt{nP}(4\sigma^2-4P\mathbb{E}^2\left(|Z|\right))}
{\zeta_{1}\zeta_{3}^2}
\nonumber\\ = &\frac{\sqrt{nP}(2\zeta_{1} \sqrt{nP}\mathbb{E}\left(|Z|\right)-4\sigma^2+4P\mathbb{E}^2\left(|Z|\right))}
{\zeta_{1}\zeta_{3}^2}>0\label{equ_39}.\end{align}
%\end{footnotesize}
Then, to analyze $\Delta_{2}$, we have

\vspace{-0.5cm}
\begin{align}&n\mathbb{V}(1-\tau)^2+(W_{ii}-\mathbb{E})^2 \nonumber
\\ =&\frac
{4nP\sigma^4+\zeta_{1}\zeta_{2}^2}{\zeta_{1}\zeta_{3}^2}
\nonumber\\ \approx &\frac{4nP\sigma^4}{\zeta_{1} (nP\mathbb{E}\left(|Z|\right))^2}\nonumber\\ \approx &
\frac{4\sigma^4}{\zeta_{1} nP\mathbb{E}^2\left(|Z|\right)}\label{equ_40}.\end{align}
Furthermore, we have

\vspace{-0.5cm}
\begin{align}\frac{\partial (P\zeta_{1})}{\partial P}=\frac{4\sigma^2-3P\mathbb{E}^2\left(|Z|\right)}{2\zeta_{1}}>0,\label{equ_41}
\end{align}
and

\vspace{-0.5cm}
\begin{align}&\frac{\partial (n\mathbb{V}(1-\tau)^2+(W_{ii}-\mathbb{E})^2)}{\partial d}\nonumber\\=&\frac{4\zeta_{1}\zeta_{3}(2d-P\mathbb{E}\left(|Z|\right))
-16nP\sigma^4-4\zeta_{1}\zeta_{2}^2}
{\zeta_{1}\zeta_{3}^3}
\nonumber\\ \approx& \frac{4nP\zeta_{1}\zeta_{2}\mathbb{E}\left(|Z|\right)
-16nP\sigma^4}
{\zeta_{1} (nP\mathbb{E}\left(|Z|\right))^3}>0,
\label{equ45}\end{align}
where the above inequality (\ref{equ45}) holds by Assumption \ref{ass2}. According to (\ref{equ_36})-(\ref{equ45}), $r$ decreases with the increasing of $n$ and $P$, and higher  self-confidence level  $d$ will decrease convergence rate.

Based on the above analysis in \textbf{(1)}-\textbf{(3)}, for the $(+/+)$, $(+/-)$ and $(+/0)$ scenarios, the convergence rate $r$ is proportional to both the population size and the network connectivity. Conversely, it is inversely proportional to the individuals' self-confidence level.

%\begin{align}r&=-log\left(|Q_{\text {uppermost }}|\right)\nonumber\\&=-log\left(\sqrt{\frac{4nP\sigma^4+(2d-P\mathbb{E}\left(|Z|\right))^2(2\sigma^2-P\mathbb{E}^2\left(|Z|\right))}{(2\sigma^2-P\mathbb{E}^2\left(|Z|\right))(2d+(n-1)P\mathbb{E}\left(|Z|\right))^2}}\right)\label{equ_8}\end{align}
%According to the form of $r$ in (\ref{equ_8}),  the convergence rate will increase with the size of individuals. Moreover, for quadratic functions $y(P)=-\mathbb{E}^2(|Z|)P^2+2\sigma^2P$ related to $P$, when $0<P\leq1$, $y(P)$ is monotonically increasing. Therefore, higher network connectivity can improve the convergence rate.
%By similar analysis as in completed Trust-Trust interactions case, the convergence rate will increase with the larger size of individuals and decrease with the higher self confidence.

%Let $\mathbb{C}=\sigma^2+\mathbb{E}^2\left(|Z|\right)$, \begin{align}r&=-log\left(|Q_{\text {uppermost }}|\right)\nonumber\\&=-log\left(\frac{\sqrt{\left(\sqrt{nP}\mathbb{C}+d\sigma\right)^2+d^2\sigma^2}}{\sigma(d+(n-1)P\mathbb{E}\left(|Z|\right))}\right)\label{equ_7}
%\nonumber\\&=-log\left(\frac{\sqrt{nP\mathbb{C}^2+2d^2\sigma^2+2\mathbb{C}d\sigma\sqrt{np}}}{\sigma(d+(n-1)P\mathbb{E}\left(|Z|\right))}\right)
%\nonumber\\&\approx-log\left(\sqrt{\frac{\mathbb{C}^2+\frac{2\mathbb{C}d\sigma}{\sqrt{nP}}}{\sigma^2nP\mathbb{E}^2\left(|Z|\right)}}\right).\end{align}

\textbf{(4). Mutual mistrust $(-/-)$ interaction.}

In this case, $|W-\text{diag}\{2(W_{ii}-\mathbb{E})\}|$ is a row-stochastic matrix and $-1$ is the dominant eigenvalue of $W-\text{diag}\{2(W_{ii}-\mathbb{E})\}$. By Lemma \ref{lem7}, $-1+2(W_{ii}-\mathbb{E})$ is an eigenvalue of $W$ and determines the convergence rate. Since $P_{-/-}=1$,
we have
$\bar{P}=-1$ and $\hat{P}=1.$ Substituting them into (\ref{equ_12}), we have

\vspace{-0.5cm} \begin{align}\label{equ_251}-1+2(W_{ii}-\mathbb{E})=\frac{d-(n-3)P\mathbb{E}\left(|Z|\right)}{d+(n-1)P\mathbb{E}\left(|Z|\right)}.\end{align} Since $|-1+2(W_{ii}-\mathbb{E})|\gg M_{e}$, we have $r=-\text{log}(|-1+2(W_{ii}-\mathbb{E})|).$ Moreover, $|-1+2(W_{ii}-\mathbb{E})|=1-2(W_{ii}-\mathbb{E})$.

%\textbf{Case I:} $|-1+2(W_{ii}-\mathbb{E})|\geq M_{e}$.

In order to further study the effect of $n$, $d$ and $P$ on convergence rate, differentiating (\ref{equ_251}) with respect to $n$, $P$ and $d$ respectively, we obtain

\vspace{-0.5cm}
\begin{align*}\frac{\partial (1-2(W_{ii}-\mathbb{E}))}{\partial n}&=\frac{2P^2\mathbb{E}^2(|Z|)+2P\mathbb{E}(|Z|)d}{{(d+(n-1)P\mathbb{E}(|Z|))}^2}>0,\end{align*}

\vspace{-0.5cm}
\begin{align*}\frac{\partial (1-2(W_{ii}-\mathbb{E}))}{\partial P}&=\frac{(2n-4)\mathbb{E}(|Z|)d}{{(d+(n-1)P\mathbb{E}(|Z|))}^2}>0,\end{align*}

\vspace{-0.5cm}
\begin{align*}\frac{\partial (1-2(W_{ii}-\mathbb{E}))}{\partial d}&=\frac{-(2n-4)d\mathbb{E}(|Z|)}{{(d+(n-1)P\mathbb{E}(|Z|))}^2}<0,\end{align*}
Thus, the increase of population size $n$ and network connectivity $P$ promote the convergence rate, while the improvement of self-confidence level $d$ decreases the convergence rate.

%$$r=-log\left(\max(a, b)\right).$$
%
%\begin{align}r&=-log\left(\max(|Q_{\text {rightmost }}|,|Q_{\text {leftmost }}|)\right)\nonumber\\&=-log\left(\sqrt{n\mathbb{V}}(1+\tau)+W_{ii}-\mathbb{E}\right)
%\nonumber\\&=-log\left(\frac{\sqrt{nP}(\sigma^2+(1-2P)\mathbb{E}^2(|Z|))+\mathbb{A}(d+P\mathbb{E}(|Z|))}{\mathbb{A}(d+(n-1)P\mathbb{E}\left(|Z|\right))}\right)\label{equ_11}\\&
%\approx -log\left(\frac{\sqrt{nP}(\sigma^2+(1-2P)\mathbb{E}^2(|Z|))}{\mathbb{A}nP\mathbb{E}\left(|Z|\right)}\right)\label{equ_14}\nonumber \\&\approx -log\left(\frac{\sigma^2+(1-2P)\mathbb{E}^2(|Z|)}{\sqrt{nP}\mathbb{A}\mathbb{E}\left(|Z|\right)}\right).\end{align}
%Clearly, more individuals can accelerate convergence. Next, we aim to analyze the effect of $d$ on rate.

%Let $$\bar{H}:=\frac{(d+P\mathbb{E}(|Z|))+\mathbb{B}}{d+(n-1)P\mathbb{E}(|Z|)},$$ where $\mathbb{B}$ and $\mathbb{A}$ are the same defined as that in Trust-Trust interactions case.
%Moreover,
%\begin{align*}\frac{\partial \bar{H}}{\partial d}=\frac{(n-2)P\mathbb{E}(|Z|)-\mathbb{B}}{(d+(n-1)P\mathbb{E}(|Z|))^2}>0,\end{align*}
%which means high self-confidence decreases convergence rate. Finally, we will study the effect of $P$ on rate.

\textbf{(5). Unilateral  mistrust $(-/0)$ interaction.}

In this case, $|W-\text{diag}\{2(W_{ii}-\mathbb{E})\}|$ is a row-stochastic matrix and  $-1$ is the dominant eigenvalue of matrix $W-\text{diag}\{2(W_{ii}-\mathbb{E})\}$. By Lemma \ref{lem7}, $-1+2(W_{ii}-\mathbb{E})$ is an eigenvalue of $W$ and determines the convergence rate. Since $P_{-/0}=1$, we have
$\bar{P}=-\frac{1}{2}$ and $\hat{P}=\frac{1}{2}.$ Substituting them into (\ref{equ_12}), we have

\vspace{-0.5cm}
\begin{align*}-1+2(W_{ii}-\mathbb{E})=\frac{2d-(n-3)P\mathbb{E}\left(|Z|\right)}{2d+(n-1)P\mathbb{E}\left(|Z|\right)}.\end{align*} Since $|-1+2(W_{ii}-\mathbb{E})|\gg M_{e}$, we have $r=-
\text{log}(|-1+2(W_{ii}-\mathbb{E})|)$ and

\vspace{-0.5cm}
\begin{align}1-2(W_{ii}-\mathbb{E})=\frac{(n-3)P\mathbb{E}\left(|Z|\right)-2d}{2d+(n-1)P\mathbb{E}\left(|Z|\right)}\label{equ_19}.\end{align}

Differentiating (\ref{equ_19}) with respect to $n$, $P$ and $d$ respectively, we obtain that the following inequalities hold

\vspace{-0.5cm}
\begin{align*}\frac{\partial (1-2(W_{ii}-\mathbb{E}))}{\partial n}&=\frac{2P^2\mathbb{E}^2(|Z|)+4Pd\mathbb{E}(|Z|)}{{(d+(n-1)P\mathbb{E}(|Z|))}^2}>0,\end{align*}

\vspace{-0.5cm}
\begin{align*}\frac{\partial (1-2(W_{ii}-\mathbb{E}))}{\partial P}&=\frac{(4n-8)d\mathbb{E}(|Z|)}{{(d+(n-1)P\mathbb{E}(|Z|))}^2}>0,\end{align*}

\vspace{-0.5cm}
\begin{align*}\frac{\partial (1-2(W_{ii}-\mathbb{E}))}{\partial d}&=\frac{(-4n+8)Pd\mathbb{E}(|Z|)}{{(d+(n-1)P\mathbb{E}(|Z|))}^2}<0.\end{align*}
 Hence, $r$ decreases with the increasing of network connectivity $P$ and population size $n$, but increases with the higher
 self-confidence level $d$.

 Based on the above analysis in \textbf{(4)} and \textbf{(5)}, for the $(-/-)$ and $(-/0)$ scenarios, the convergence rate $r$ is inversely proportional to both the population size and the network connectivity. On the contrary, it is proportional to the individuals' self-confidence level.
 \end{proof}

 From Theorems \ref{the1}-\ref{the3}, the impact that the $(-/-)$ and $(-/0)$ have on the convergence rate is completely opposite to that of other interaction types. To further explore the role of mistrust in the process of opinion evolution, we proceed to analyze two mixed scenarios: $(+/+,-/-)$ and $(-/-,-/0)$.

\subsection{Mixture interaction $(+/+,-/-)$}
In this section, we focus on studying the effect of mutual trust $(+/+)$ on the convergence rate of system (\ref{equ001}) by considering the mixture of mutual trust $(+/+)$ and mutual mistrust $(-/-)$.
\begin{theorem}\label{the4}
 \textup{For the system (\ref{equ001}) with mixture interactions $(+/+,-/-)$, there exist two small constants $\xi_{1},\xi_{2}>0$ such that the following statements hold:}
  \\
    \noindent
  \textup{ (i). When $P_{+/+}$ is in the intervals $[0, 0.5-\xi_{1})$ or $(0.5,0.5+\xi_{2}]$, the convergence rate  of system (\ref{equ001}) is proportional to $P_{+/+}$.}
  \\
  \noindent
   \textup{(ii). When $P_{+/+}$ is in the intervals $[0.5-\xi_{1},0.5]$ or $(0.5+\xi_{2},1]$, the convergence rate  of  system (\ref{equ001}) is inversely proportional to $P_{+/+}$.}
\end{theorem}
\begin{proof}
 Substituting $\hat{P}=1$, $\bar{P}=2P_{+/+}-1$, and $P_{+/-}=P_{+/0}=P_{-/0}=0$ into (\ref{equ_12}), we obtain
  \begin{equation*}\left\{\begin{array}{l}
\mathbb{E}=\frac{\mathbb{E}\left(S_{ij}\right)}{C_{m}+d}=\frac{P\bar{P}\mathbb{E}\left(|Z|\right)}{(n-1)P\mathbb{E}\left(|Z|\right)+d}, \\
W_{ii}-\mathbb{E}=\frac{d-P\bar{P}\mathbb{E}\left(|Z|\right)}{(n-1)P\mathbb{E}\left(|Z|\right)+d}\\
\sqrt{n\mathbb{V}}=\frac{\sqrt{nP\sigma^2-nP^2\bar{P}^2\mathbb{E}^2\left(|Z|\right)}}{(n-1)P\mathbb{E}\left(|Z|\right)+d},\\
\tau=\frac{\mathbb{E}^2\left(|Z|\right)-P\bar{P}^2\mathbb{E}^2\left(|Z|\right)}{\sigma^2-P\bar{P}^2\mathbb{E}^2\left(|Z|\right)},\\
\lambda_{\text{outlier}}=\frac{\mathbb{E}^2\left(|Z|\right)+(n-2)P\bar{P}^2\mathbb{E}^2\left(|Z|\right)+d\bar{P}\mathbb{E}\left(|Z|\right)}{\bar{P}\mathbb{E}\left(|Z|\right)((n-1)P\mathbb{E}\left(|Z|\right)+d)}.
\end{array}\right.\end{equation*}
 From Theorem \ref{the2}, when $P_{+/+}\in[0, 0.5-\xi_{1})\cup(0.5+\xi_{2}, 1]$, there exist two small constants $\xi_{1},\xi_{2}>0$ such that $r=-\text{log}(|\lambda_{\text{outlier}}|)$. Moreover, $|\lambda_{\text{outlier}}|\approx |\bar{P}|$. Thus, when $P_{+/+}\in[0, 0.5-\xi_{1}]$, a larger $P_{+/+}$ can promote convergence rate $r$, while when $P_{+/+}\in[0.5+\xi_{2}, 1]$, conclusions are just the opposite.

 Moreover, when $P_{+/+}\in[0.5-\xi_{1},0.5]\cup[0.5,0.5+\xi_{2}]$, we have \[
 r=-\text{log}(\sqrt{n\mathbb{V}}(1+\tau)+|W_{ii}-\mathbb{E}|)\]
  and \begin{align*} &\sqrt{n\mathbb{V}}(1+\tau)+|W_{ii}-\mathbb{E}|\\ =
&\frac{\sqrt{nP}(\theta+\mathbb{E}^2\left(|Z|\right)-P\bar{P}^2\mathbb{E}^2\left(|Z|\right))+\theta|d-P\bar{P}\mathbb{E}\left(|Z|\right)|}
{(n-1)P\theta\mathbb{E}\left(|Z|\right)}
\\ \approx&\frac{\theta+\mathbb{E}^2\left(|Z|\right)-P\bar{P}^2\mathbb{E}^2\left(|Z|\right)}
{\sqrt{nP}\theta \mathbb{E}\left(|Z|\right)}
\\=&\frac{1+\frac{\mathbb{E}^2\left(|Z|\right)-P\bar{P}^2\mathbb{E}^2\left(|Z|\right)}{\theta}}
{\sqrt{nP}\mathbb{E}\left(|Z|\right)},\end{align*}
where $\theta=\sqrt{\sigma^2-P\bar{P}^2\mathbb{E}^2\left(|Z|\right)}$. Furthermore, since \begin{align*}&\frac{\partial (\frac{\mathbb{E}^2\left(|Z|\right)-P\bar{P}^2\mathbb{E}^2\left(|Z|\right)}{\theta})}{\partial (\bar{P}^2)}\\=&
 \frac{-2P\mathbb{E}^2\left(|Z|\right)\sigma^2+P^2\bar{P}^2\mathbb{E}^4\left(|Z|\right)
 +P\mathbb{E}^4\left(|Z|\right)}{2\theta^3}<0,\end{align*} the convergence rate $r$ monotonically decreases and increases in intervals $[0.5-\xi_{1},0.5]$ and $(0.5,0.5+\xi_{2}]$, respectively.
 \end{proof}
\begin{remark}
%In Theorem \ref{the2}, if Assumption \ref{ass1} holds, $\rho(W)<1$ and $|\lambda_{\text{outlier}}|<1$. Hence, When $|n\mathbb{E}|>\sqrt{n\mathbb{V}}$, convergence rate $r$ is well defined.
\textup{From Theorem \ref{the4}, more trust interaction types $((+/+)$ and $(+/0))$ do not necessarily accelerate convergence. In ecology, this phenomenon is analogous to mutualism among multiple species, while in sociology, it is indicative of social balance.}
\end{remark}
 \subsection{Mixture interaction $(-/-,-/0)$}
In this section, we examine the effect of increasing mutual interaction in system (\ref{equ001}) by considering the mixture $(-/-,-/0)$ of mutual mistrust $(-/-)$ and unilateral $(-/0)$.
\begin{theorem}\label{the5}\textup{
For the system (\ref{equ001}) with mixture interaction $(-/-,-/0)$, the convergence rate is inversely proportional to the proportion of mutual mistrust $(-/-)$.}
\end{theorem}

\begin{proof}
For the mixture interaction $(-/-,-/0)$, we have \[
\hat{P}=-\bar{P}=\frac{1}{2}(P_{-/-}+1).\] By Theorem \ref{the3}, we can obtain
\[
r=-\text{log}(1-2(W_{ii}-\mathbb{E})),
\] where

\vspace{-0.5cm} \begin{align}1-2(W_{ii}-\mathbb{E})=\frac{(n-3)P\hat{P}\mathbb{E}\left(|Z|\right)-d}
{(n-1)P\hat{P}\mathbb{E}\left(|Z|\right)+d}.\label{equ_17}\end{align}

Then, differentiating (\ref{equ_17}) with respect to $\hat{P}$ yields

\vspace{-0.5cm}
\begin{align*}\frac{\partial (1-2(W_{ii}-\mathbb{E}))}{\partial \hat{P}}&=\frac{(2n-4)Pd\mathbb{E}(|Z|)}{{(d+(n-1)P\hat{P}\mathbb{E}(|Z|))}^2}>0.\end{align*}
Thus, the convergence rate is inversely proportional to the proportion of $(-/-)$.
\end{proof}
%\begin{figure}
%\centering
%\includegraphics[scale=0.43]{fig9b.pdf}
%\caption{Convergence rate of system (\ref{equ001}) with  mixture interaction $(-/-,-/0)$. Parameters $n$, $P$, $d$ are the same as those in Fig. \ref{pic-4}. Less mutual mistrust $(-/-)$ interaction accelerates the convergence rate of system (\ref{equ001}).}
%\label{pic-11}
%\end{figure}

\begin{remark}\textup{
In this paper, we introduce a novel framework for analyzing the convergence rate of the discrete-time Altafini model and identify the key factors that influence this rate. Most importantly, our research method is effective not only for the discrete-time Altafini model but also for the continuous-time Altafini model as discussed in \cite{Altafini:13} and the opinion dynamics model with stubborn individuals as presented in \cite{Fri16}.}
\end{remark}%(see Fig. \ref{tab1})
\section{Illustrative examples}\label{main2}
In this section, some numerical examples are given to illustrate our main results.
\begin{example}\label{exam1}\textup{(Random mixture interactions) This example aims to show the validity of the theoretical results in Theorem \ref{the1} and Corollary \ref{cor1}, where the interaction strengths $S_{ij}$ are drawn from a normal distribution with mean $\mathbb{E}(Z)=0$ and variance $\sigma=1$.}
\par
\textup{(1). Let the parameter values $n$ be 500 and $d$ be 5, respectively. Then, by Theorem \ref{the1}, we estimate the eigenvalue distribution of matrix $W$ in (\ref{equ001}). As shown in FIG. \ref{pic-random}, results with different parameter values $P$ from numerical simulations align well with the theoretical estimations.}
\begin{figure}[ht]
    \centering
    \includegraphics[scale=0.25]{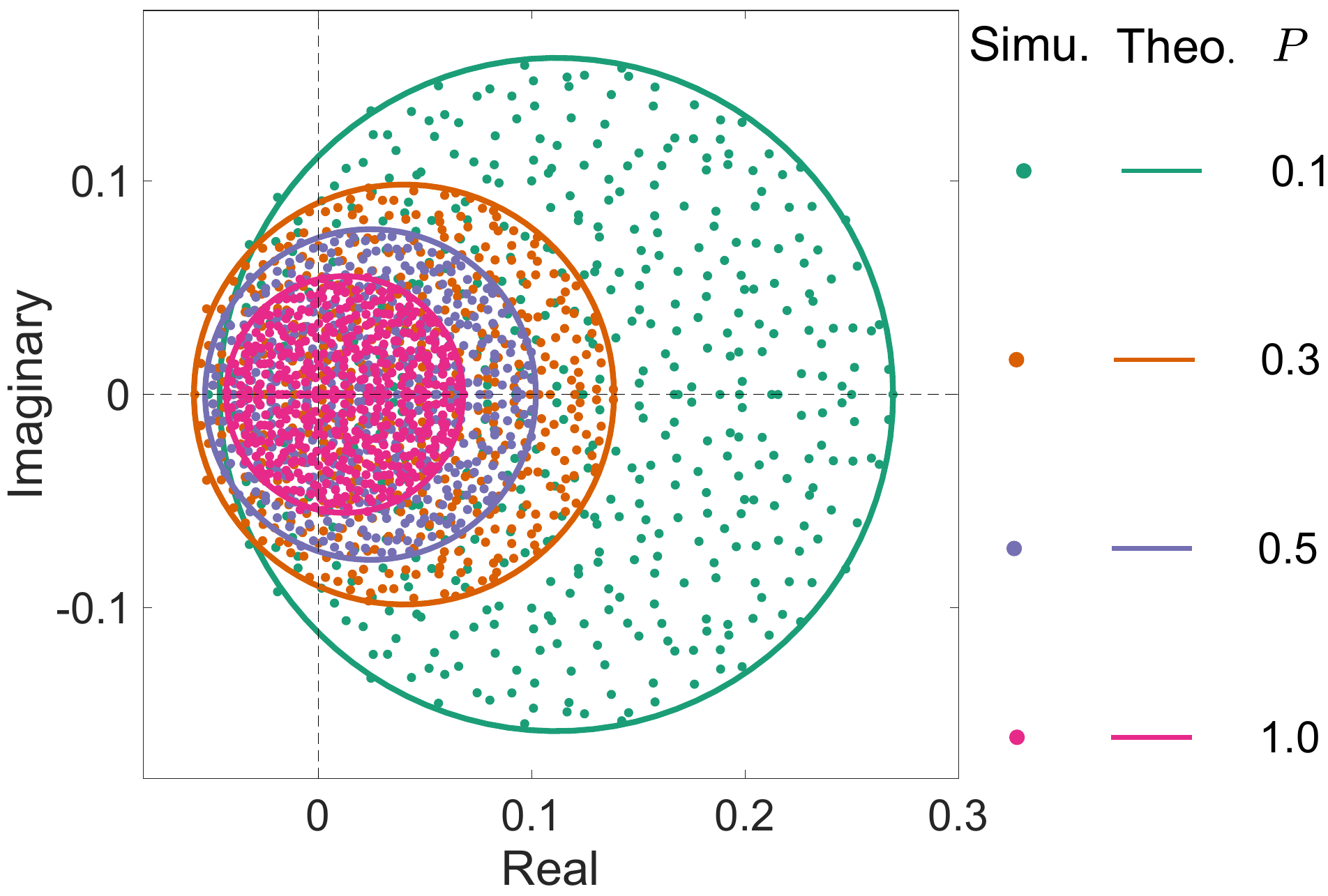}
    \caption{Estimating the eigenvalue distribution of $W$. Solid line and dot are obtained from our theory (Theo.) results and numerical simulations (Simu.), respectively.}
    \label{pic-random}
    \end{figure}
 \par
\textup{(2). Next, we verify the effects of population size, network connectivity, and individuals' self-confidence level on the convergence rate of the system (\ref{equ001}), where the parameter values $n$, $P$, and $d$ are set in TABLE \ref{table1}. From Fig. \ref{fig8b}, the results from numerical simulations are in good agreement with our theoretical analysis.}
\begin{table}[ht]
\centering
    \fontsize{7}{10}\selectfont    %{×ÖÌå³ß´ç}{Ðоà}
    \caption{Parameter values in Example \ref{exam1}}
\begin{tabular}{c|c|c|c}
    \hline
      & Fig. \ref{fig8b} $(a)$ & Fig. \ref{fig8b} $(b)$ & Fig. \ref{fig8b} $(c)$  \\
    \hline
     Population size $n$& 100:100:1500 & 500 & 500 \\
    \hline
   Network connectivity $P$ & 0.5  & 0.1:0.05:0.9 & 0.5 \\
    \hline
   Self-confidence level $d$ & 5  & 5 & 10:10:100 \\
    \hline
\end{tabular}
\label{table1}
\end{table}
   \begin{figure*}[ht]
    \includegraphics[scale=0.4]{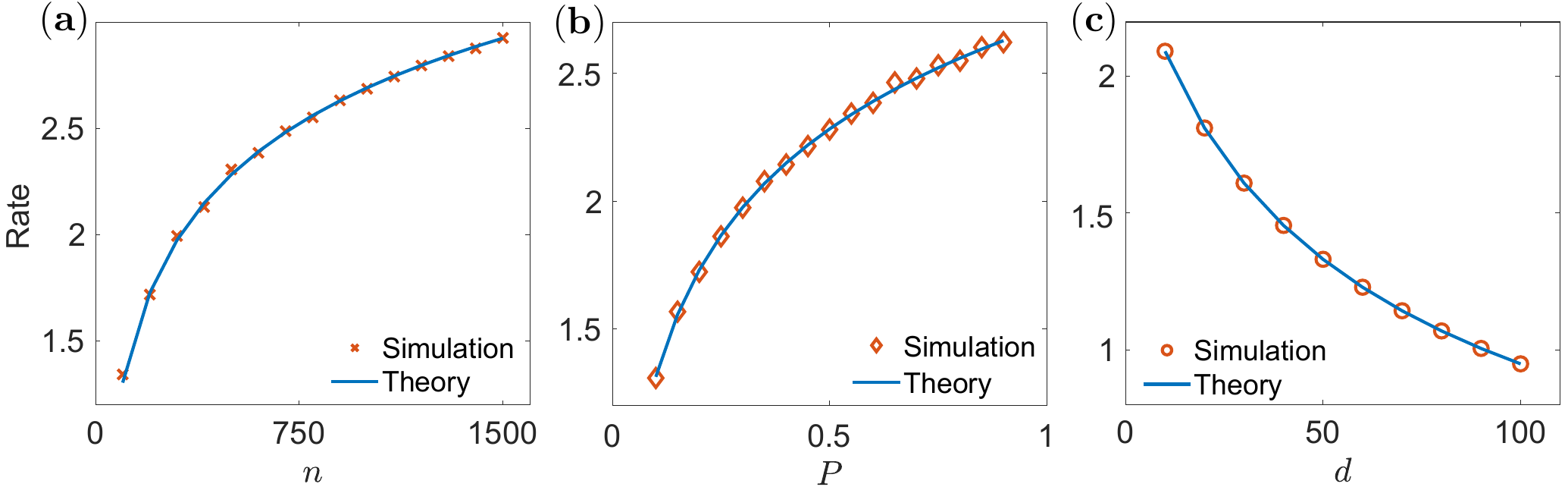}
    \centering
    \caption{Convergence rate of system (\ref{equ001}) with random mixture interactions.}
    \label{fig8b}
  \end{figure*}
\end{example}
\begin{example}\label{exam2}\textup{(Complex mixture interactions)
This example aims to demonstrate the validity of the results on the convergence rate presented in Theorems \ref{the2} and \ref{the3}. The interaction strengths $S_{ij}$ are still drawn from a normal distribution with mean $\mathbb{E}(Z)=0$ and variance $\sigma=1$.}
\par
\textup{(1). Set the parameter values to $n=500$, $d=5$, and $P=0.5$. It is easy to verify that Assumption \ref{ass2} holds. Next, by Theorems \ref{the2}, we estimate the eigenvalue distribution of matrix $W$ in (\ref{equ001}). As shown in Fig. \ref{pic-4}, we find that the results from numerical simulations closely match the theoretical estimations.}
\begin{figure*}[ht]
    \centering
    \includegraphics[scale=0.4]{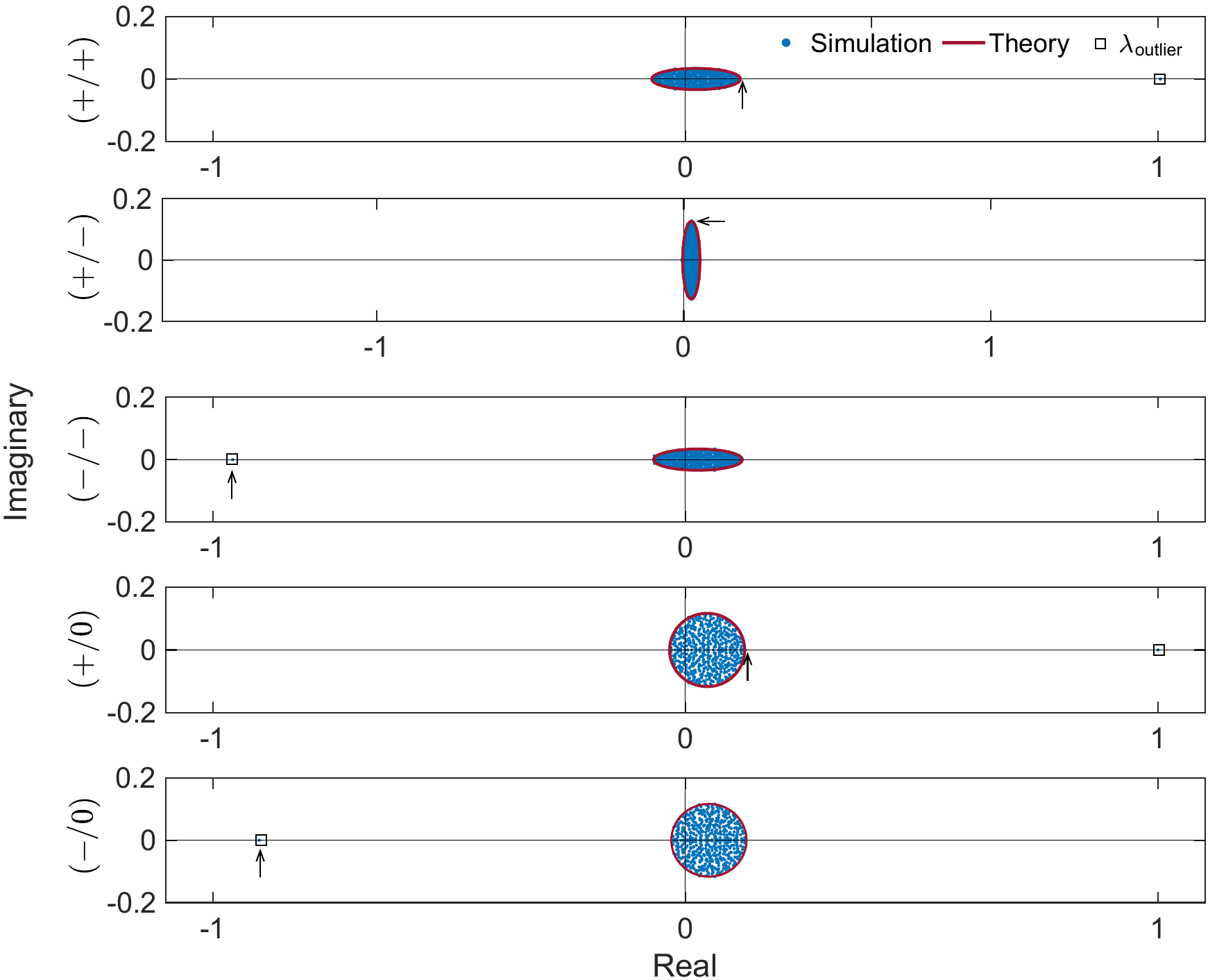}
    \caption{The eigenvalue distribution of $W$ under five interaction scenarios: mutual trust $(+/+)$, mutual mistrust $(-/-)$, trust$/$mistrust $(+/-)$, unilateral trust $(+/0)$, and unilateral mistrust $(-/0)$. The arrow points towards the eigenvalue $\lambda_{i}(W)$ that governs the convergence rate of the system (\ref{equ001}). For mutual trust $(+/+)$ and unilateral trust $(+/0)$, $\lambda_{\text{outlier}}=1$. For mutual mistrust $(-/-)$ and unilateral mistrust $(-/0)$, $\lambda_{\text{outlier}}=1-2W_{ii}$. In the case of trust$/$mistrust $(+/-)$, all eigenvalues are distributed in an ellipse.}
  \label{pic-4}
  \end{figure*}
  \par
\textup{ (2). Now, we verify the effects of population size, network connectivity, self-confidence level, and complex interaction types on the convergence rate of the system (\ref{equ001}), where the parameter values $n$, $P$, and $d$ are provided in TABLE \ref{table2}. As depicted in FIG. \ref{fignpd}, for the $(+/+)$, $(+/-)$, and $(+/0)$ scenarios, an increase in population size and network connectivity results in faster convergence, whereas a higher self-confidence level slows the rate. Conversely, for the $(-/-)$ and $(-/0)$ scenarios, a higher self-confidence level accelerates convergence. Hence, these numerical simulation results are in accordance with our theoretical predictions.}
 \begin{table}[ht]
\centering
    \fontsize{7}{10}\selectfont    %{×ÖÌå³ß´ç}{Ðоà}
    \caption{Parameter values in Example \ref{exam2}}
\begin{tabular}{c|c|c|c}
    \hline
      & Fig. \ref{fignpd} $(a)$ & Fig. \ref{fignpd} $(b)$ & Fig. \ref{fignpd} $(c)$  \\
    \hline
     Population size $n$& 50:50:1500 & 500 & 500 \\
    \hline
   Network connectivity $P$ & 0.5  & 0.05:0.05:1 & 0.5 \\
    \hline
   Self-confidence level $d$ & 5  & 5 & 3:4:47 \\
    \hline
\end{tabular}
\label{table2}
\end{table}
 \begin{figure*}[ht]
    \centering
    \includegraphics[scale=0.33]{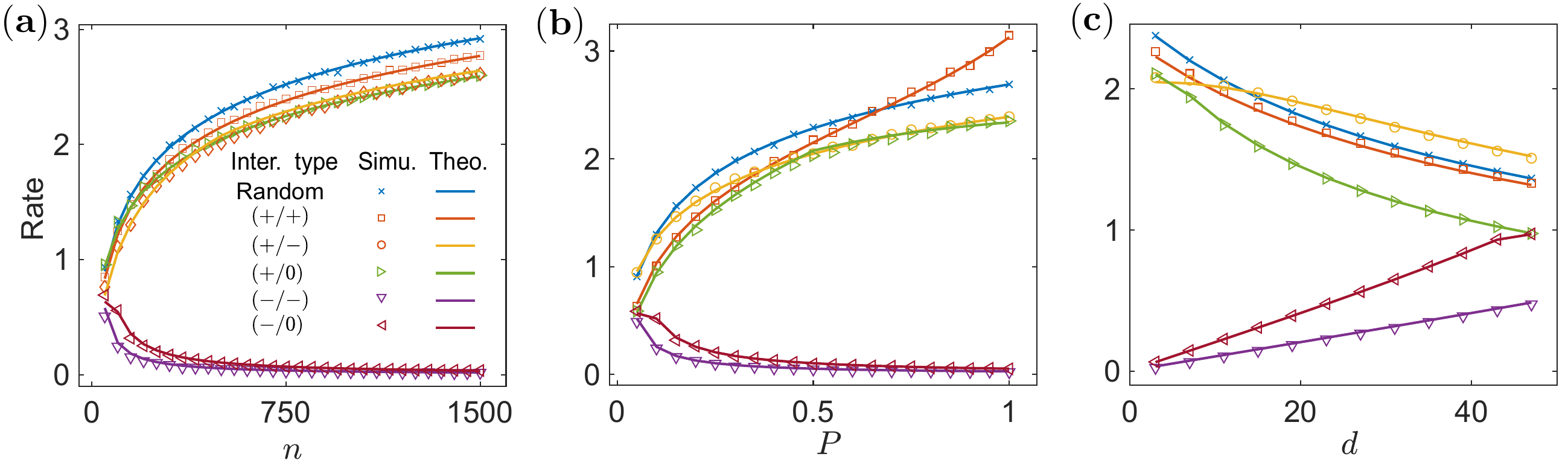}
    \caption{Convergence rate of system (\ref{equ001}) under six interaction scenarios: random mixture, mutual trust $(+/+)$, trust$/$mistrust $(+/-)$, unilateral trust $(+/0)$, mutual mistrust $(-/-)$, and unilateral mistrust $(-/0)$.}
\label{fignpd}
  \end{figure*}
\end{example}
\begin{example}\textup{$($Two mixture interaction scenarios: $(+/+,-/-)$ and $(-/-,-/0))$
In this example, we focus on examining the optimal proportional configuration to ensure the fastest convergence rate for the interactions scenarios discussed in Theorems \ref{the4} and \ref{the5}. Let the parameter values be $n=500$, $d=5$, and $P=0.5$.}

 \textup{As depicted in FIG. \ref{pic-7}, for the $(+/+,$ $-/-)$, $(+/+,+/-)$, $(+/+,-/0)$ scenarios, the convergence rate is not monotonic as the proportion of the $(+/+)$ interaction type increases. Specifically, for the scenario of interaction type $(+/+,-/-)$, when the proportion of $(+/+)$ is approximately 0.5 $(P_{+/+}\approx0.5)$, system (\ref{equ001}) achieves convergence in the fastest speed. From FIG. \ref{pic-7} $(b)$, for the $(-/-,-/0)$ scenario, the convergence rate is inversely proportional to the proportion of $(-/-)$ interaction types.}
\begin{figure*}[ht]
\centering
\includegraphics[scale=0.4]{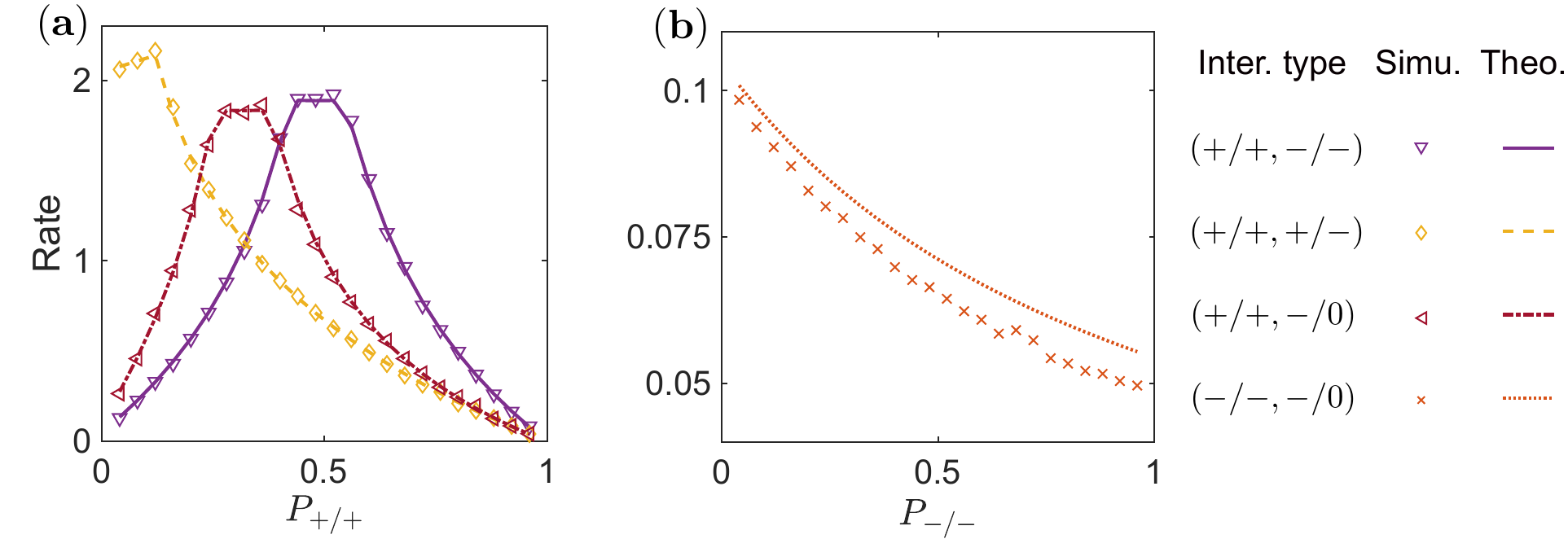}
\caption{Convergence rate of system (\ref{equ001}) under four mixture interaction scenarios: $(+/+,-/-)$, $(+/+,+/-)$, $(+/+,-/0)$, and $(-/-,-/0)$.}
\label{pic-7}
\end{figure*}
\end{example}
\section{Conclusion}\label{con1}
A general framework for the analysis of the convergence rate is established in this paper. Firstly, we have established the quantitative expressions of convergence rate by random matrix theory and low-rank perturbation theory. These results bridge the gap between the convergence rate and complex interaction types. With the aid of this bridge, we have further analyzed the impact of some key factors on the convergence rate through rigorous theoretical derivations. In addition to theoretical analyses, we have also provided simulation examples to corroborate our findings, thereby demonstrating the significant impact of interaction types on the convergence rate.

In a realistic social network, the information transmitted among the individuals may be subject to communication constraints such as delays from time to time, ranging from engineering science (distributed control \cite{he1,he2}) to ecosystems (ecological stability \cite{yang}), and social sciences (opinion forming \cite{zuo}). From the perspective of opinion dynamics, a communication delay between a pair of individuals represents that one individual can only access an earlier opinion of the other, which leads to the fact that individuals cannot express their opinions in a precise manner. This naturally raises a significant question: what impact does time delay have on the convergence rate of opinion dynamics? Whether our approach in this paper can be extended to such a situation remains open for future investigations.

%\bibliography{reference}Õâ¸öÊÇ·ÅbibÎļþµÄ
\end{document}